\documentclass[a4paper,11pt]{amsart}
\usepackage{latexsym,amssymb,amsthm,amsmath,color}
\usepackage{xspace}
\usepackage{tikz-cd}
\usepackage{mathtools}
\usepackage[all]{xy}
\usepackage[abs]{overpic}
\usepackage{hyperref}
\makeatletter
\renewcommand\subsubsection{\@startsection{subsubsection}{3}%
	\z@{.5\linespacing\@plus.7\linespacing}{-.5em}%
	{\normalfont\bfseries}}
\makeatother
\newcommand{\cc}{\ensuremath{\mathbb{C}}\xspace}

\newcommand{\rr}{\ensuremath{\mathbb{R}}\xspace}
\newcommand{\zz}{\ensuremath{\mathbb{Z}}\xspace}
\newcommand{\D}{\ensuremath{\mathbb{D}}\xspace}
\newcommand{\Cour}[1]{[\![ #1 ]\!]}
\newcommand{\abs}[1]{\left| #1 \right|}

\newcommand{\set}[1]{\lbrace #1 \rbrace}
\newcommand{\comp}{\hspace{0.1 mm} \circ \hspace{0.1 mm}}

\renewcommand{\bar}{\overline}
\newcommand{\alert}[1]{{\color{red} #1}}
\newcommand{\inp}[1]{\left\langle #1 \right\rangle}

\DeclareMathOperator{\Res}{Res}

\newcommand{\A}{\mathcal{A}}
\newcommand{\elli}{\mathcal{A}_{\abs{D}}}
\newcommand{\ellio}{\mathcal{A}_{\abs{D}}}
\newcommand{\ellit}{\mathcal{A}_{\abs{\hat{D}}}}
\newcommand{\elliti}{\mathcal{A}_{\abs{\tilde{D}}}}
\newcommand{\ellihat}{\mathcal{A}_{\abs{\hat{D}}}}
\newcommand{\J}{\mathcal{J}}

\DeclareMathOperator{\im}{Im}

\DeclareMathOperator{\bas}{bas}
\DeclareMathOperator{\red}{red}

\DeclareMathOperator{\Sing}{Sing}

\renewcommand{\phi}{\varphi}
\usepackage{thmtools}
\declaretheorem[style=definition,qed=$\diamondsuit$]{definition}
\declaretheorem[style=definition,qed=$\triangle$,sibling=definition]{example}
\declaretheorem[style=plain,sibling=definition]{theorem}
\declaretheorem[style=plain,sibling=definition,numbered=no,name=Theorem]{theoremnonum}
\declaretheorem[style=plain,sibling=definition]{lemma}
\declaretheorem[style=plain,sibling=definition]{proposition}
\declaretheorem[style=plain,numbered=no,sibling=definition,name=Proposition]{propositionnonum}
\declaretheorem[style=plain,sibling=definition]{corollary}
\declaretheorem[style=definition,qed=$\diamondsuit$,sibling=example]{claim}
\declaretheorem[style=definition,qed=$\diamondsuit$,sibling=claim]{remark}
\newtheorem*{claim*}{Claim}
\numberwithin{theoremalpha}{section}

\numberwithin{equation}{section}
\numberwithin{definition}{section}
\numberwithin{theorem}{section}
\numberwithin{proposition}{section}
\numberwithin{lemma}{section}
\numberwithin{example}{section}
\numberwithin{remark}{section}
\numberwithin{corollary}{section}

\newcommand{\locm}{\cc \times_{\rr_{\geq 0}} \cc}

\definecolor{ao(english)}{rgb}{0.0, 0.5, 0.0}
\newtheoremstyle{named}{}{}{\itshape}{}{\bfseries}{.}{.5em}{\thmnote{#3}#1}
\theoremstyle{named}
\newtheorem*{namedtheorem}{}
\numberwithin{equation}{section}
\title{Non-principal $T$-duality, generalized complex geometry and blow-ups}

\author{Gil R. Cavalcanti}
\address{Department of Mathematics, Utrecht University, 3508 TA Utrecht, The Netherlands}
\email{g.r.cavalcanti@uu.nl}
\author{Aldo Witte}
\address{Departement of Mathematics, Celestijnenlaan 200B, B-3001 Leuven (Heverlee), Belgium}
\email{aldowitte@hotmail.nl}

\begin{document}
\begin{abstract}
We extend the notion of $T$-duality to manifolds endowed with non-principal torus actions. The singularities of the torus action are controlled by a certain Lie algebroid, called the elliptic tangent bundle. Using this Lie algebroid, we explain how certain invariant generalized complex structures can be transported via $T$-duality. Along the way, we use the elliptic tangent bundle to define connections for these torus action, and give new insight to the classification of such actions by Haefliger-Salem.
\end{abstract}
\maketitle

\begin{center}
\emph{Dedicated to the memory of Gilles Castel}
\end{center}

\tableofcontents
\section{Introduction}
In this paper we will study three topics: $T$-duality for torus actions with fixed points, blow-ups of elliptic divisors and will apply those two to the study of generalized complex structures. While each one of these has a life on its own, the interplay between them makes them all the more interesting. 

\subsection{$T$-duality}
\emph{$T$-duality} is an equivalence between quantum field theories with very different properties. Its origins trace back to a paper by Sathiapalan \cite{PhysRevLett.58.1597} where it was noted that different compactifications with $S^1$-symmetry were related by, among other things, inversion of the radius of the circles. This relation is frequently referred to as taking ``dual circles'', ``exchange of tangent and cotangent directions'' or ``swap of winding and momentum''.  The precise relationship between (local) $T$-dual Riemannian structures was first understood by Buscher in \cite{Bus87} and was developed further by Ro\v{c}ek and Verlinde in \cite{RV92}. A global notion of $T$-duality was formulated in \cite{BEM04,BHM04} as an equivalence between {\it principal torus bundles}. In \cite{CG11} it was observed that $T$-duality can be understood as an isomorphism of Courant algebroids, therefore formalizing the operation ``exchange of tangent and cotangent directions'' and allowing for a systematic transport of geometric structures between $T$-dual spaces. Because this well established notion of $T$-duality only holds for principal torus bundles, we refer to it as {\it principal $T$-duality}.

Once we move away from principal circle bundles and start considering circle actions with fixed points problems quickly pop up. Namely, the inversion of radii, a hallmark of $T$-duality, implies that every fixed point on one side is matched by an open end of a manifold on the $T$-dual side. Performing $T$-duality this way amounts to removing the fixed points and only considering the resulting open manifold where the action is free. For example, taking the 2-sphere with rotation along the $z$-axis, the dual becomes an open cylinder with an incomplete metric. Versions of this phenomenon appear in several different places, most notably, when trying to produce SYZ-mirrors of Fano varieties \cite{MR2257391,MR3415066,hori2000mirror}. Further, the strategy of ``simply ignore fixed points'' may leave one with the impression that we are missing important behaviour that takes place ``at infinity''.

More ellaborate strategies that involve the geometry and nature of the fixed point set were considered by Bunke and Schick for actions with finite isotropy groups \cite{MR2246781}, Mathai and Wu for general circle actions \cite{MR2989461} and further extended by Pande \cite{MR3952356}. In all of these cases, topological results associated to T-duality were re-obtained, but the candidate for T-dual is often a singular space (orbi-spaces, stratified singular spaces or stacks). The issue of transfer of geometric structures (instead of topological) between T-dual spaces is not addressed because the techniques used are topological in nature and the candidate ``true'' T-dual space may even have higher dimension than the original one.

The main goal of the current article is to develop a new framework for $T$-duality, with the following properties:
\begin{itemize}
\item allows for torus actions which locally decompose as standard rotations of the plane (we call such actions {\it standard}, see Definition \ref{def:standard action}),
\item is a duality between smooth spaces,
\item transports invariant geometric structures compatible with the torus action,
\item has the usual properties of T-duality  such as isomorphism of twisted cohomology and Courant algebroids.
\end{itemize}

The new input that allows us to study these actions without compromising on compactness  or smoothness of the $T$-dual are \emph{elliptic divisors} and \emph{elliptic tangent bundles}. An elliptic divisor is an ideal, $I_{\abs{D}}$, of smooth functions, which comes naturally associated to a standard torus actions. The elliptic tangent bundle is a Lie algebroid,  $\elli$, defined as the vector fields which preserve the elliptic divisor. 
 
As with principal $T$-duality, $T$-duality for these actions (Definition \ref{def:t-duality general}) will provide an isomorphism of Courant algebroids, but now using the elliptic tangent bundle:
\begin{namedtheorem}[Theorem \ref{th:courantisoelliptic}]
Let $M,\hat{M}$ be manifolds with standard $T^k$ actions and let $\elli \rightarrow M, \ellihat \rightarrow \hat{M}$ denote the induced elliptic tangent bundles. Let $H \in \Omega^3_{T^k}(\elli),\hat{H} \in \Omega^3_{T^k}(\ellihat)$ be closed invariant forms. If $(M,H)$ and $(\hat{M},\hat{H})$ are $T$-dual then there exists an isomorphism of Courant algebroids:
\begin{equation}
(\elli\oplus \elli^*)/T^k \overset{\simeq}{\longrightarrow} (\ellihat \oplus \ellihat^*)/T^k.
\end{equation}
\end{namedtheorem}
This isomorphism of Courant algebroids allows us to transport geometric structures which are invariant under the torus action between on $\elli$ over $M$ and $\ellihat$ over $\hat{M}$. As we will see later, this has implications for $T$-duality of complex manifolds as well as further special classes of generalized complex manifolds.

%Notice that we phrased $T$-duality as a relation because even for principal torus bundles $T$-duals are not unique \cite{MR2130624}. The presence of fixed points provides another source non-uniqueness (see Section \ref{sec:topchange}).

Already in the case of principal torus bundles, $T$-duals may not exist \cite{BHM04}. In that case, the proof of existence is constructive and involves choices of connections and a switch between the Chern class of the torus bundle with a collection of cohomology classes determined by the background 3-form. To carry out a similar construction of $T$-duals for torus actions with fixed points, one first needs to develop the theory of connections and curvature for those actions.

\subsubsection{Torus actions and connections}
An integral part of this paper is the classification of torus actions on manifolds. For principal torus actions, the relevant datum is its characteristic class. For general torus actions, this question is studied by Haefliger and Salem in \cite{HS91}. They show that the action is classified by its local model and a class in $\check{H}^2(M;\mathbb{Z}^k)$. The class in $\check{H}^2(M;\mathbb{Z}^k)$ has a geometric interpretation: it is the obstruction to finding a global cross-section for the quotient map $M \to M/T^k$.

Although the class of \cite{HS91} resembles a Chern class, no connection approach exists. In this article we will show that one may use the elliptic tangent bundle to give a definition of connections for standard $T^k$-actions. The curvature of these connections then induces the class of Haefliger-Salem in the following way (see Proposition \ref{prop:comparison} for a precise statement):
\begin{propositionnonum}
Let $M, M'$ be manifolds endowed with standard $T^k$-actions and common orbits space $B$ for which the torus actions are locally equivalent. Let $\Theta,\Theta'$ be elliptic connection one-forms on $M$ respectively $M'$. Then the class $[d\Theta-d\Theta']  \in H^2(A_{\partial B};\mathfrak{t}^k)$ (the logarithmic cohomology of the base with values in the Lie algebra of $T^k$) is representable by a smooth form, and descends to the Haefliger-Salem class of $M'$ in $\check{H}^2(B;\mathfrak{t}^k)$.
\end{propositionnonum}

With the notions of connection and curvature at hand we can move on to produce a construction of a $T$-duals, therefore establishing conditions that ensure their existence (c.f. Theorem \ref{prop:Texistence}). The $T$-dual we construct has a $T^k$-action which is locally equivalent to the original one. Yet we also observe that there are other $T$-duals to $(M,H)$ which do not satisfy this property, hence the fixed point set gives rise to a new source of non-uniqueness of $T$-duals.

%To establish the theorem on existence of $T$-duals we need to take a detour 
%
%and when they do, they are not unique \cite{MR2130624}. Also the presence of fixed points provides another source non-uniqueness. With those caveats in mind, in favorable circumstances one can still talk about existence, which leads to classification and connections...
%
%and when they do, they are not unique \cite{MR2130624}.

\subsection{Blow-ups} Blow-ups come from complex/algebraic geometry as a method to resolve singularities, but have encountered applications in the study other geometric structures specially as a way to construct new examples \cite{MR772133,MR2511043,BCD16}.

The minimum amount of data needed to construct canonically the blow-up of a space at a point is a holomorphic ideal for that point \cite{BCD16}. As it is, holomorphic ideals, $\mathcal{I}_D$, are the heart and soul of complex log tangent bundles. Further, there is a close connection between complex log and elliptic tangent bundles: complex log tangent bundles are in correspondence with elliptic tangent bundles together with coorientations \cite{CG17}. So, after a finite set of choices, there is a canonical blow-up construction for manifolds with elliptic divisors. Not only that but the blow-up itself inherits canonically an elliptic divisor and in many ways the blow-up manifold is indistinguishable from the original one (see Theorem \ref{th:blowupelliptic} for a precise statement).

\begin{theoremnonum}
Let $I_{\abs{D}}$ be an elliptic divisor, and let $\elli$ be the associated elliptic tangent bundle. Let $Y \subset D$ be a submanifold compatible with $D$ in a suitable sense. Then after a choice of coorientations there is a canonical blow-up $\tilde{M}$ of $M$ along $Y$ and there exists a natural elliptic divisor $I_{\abs{\tilde{D}}}$ on $\tilde{M}$ such that the blow-down map $p : \tilde{M} \rightarrow M$ provides a Lie algebroid submersion
\begin{equation*}
p : \elliti \rightarrow \elli.
\end{equation*}
\end{theoremnonum}
This means in particular that any structure on $\mathcal{A}_{|D|}$ can be pulled back to $\tilde{\mathcal{A}}_{|\tilde{D}|}$.

When $M$ comes equipped with a torus action and $Y = \set{*}$ is a fixed point of this action, the blow-up $\tilde{M}$ naturally inherits a torus action and we can compare the constructions of $T$-duality and blow-up:
\begin{namedtheorem}[Proposition \ref{prop:blow-upT}]
$T$-duality and blowing up commute.
\end{namedtheorem}
This procedure will allow us to construct many examples of $T$-dual manifolds and transport geometric structures between them.

\subsection{Generalized complex structures}
As an illustration of how to transport structures between $T$-dual spaces, we will study how generalized complex structures behave in this framework. The easy result here is that Theorem \ref{th:courantisoelliptic} allows us to transport invariant structures between $\elli\oplus \elli^*$ and $\ellihat\oplus \ellihat^*$, which is directly analogous to what happens with principal $T$-duality. The new insight here is that this transport in some cases may happen in disguise and associated to a {\it lifting} question.

One interesting phenomenon we encounter is the following: Given a manifold with elliptic divisor $(M,\mathcal{I}_{|D|})$ and a generalized complex structure $\J$ on $M$, under favourable conditions we can lift $\J$ canonically to a generalized complex structure $\J_\mathcal{A}$ on $\elli$. In some cases, this switch can be subtle enough that goes unnoticed. For example, if $(M,I)$ is a complex manifold and $\mathcal{I}_{|D|}$ is the elliptic divisor associated to a complex hypersurface $D \subset M$, then $I$ lifts to a complex structure on $\elli$. At the other extreme, the type of structure that lifts to a symplectic structure on the elliptic tangent bundle is a stable generalized complex structure on $M$ \cite{CG17,CKW20} and this switch is very visible because stable structures would normally be recognised as singular symplectic structures on $M$.

Pairing this remark with the framework we developed for $T$-duality, we produce examples with the following characteristics: we look for a manifold with a torus action compatible elliptic divisor $(M,\mathcal{I}_{|D|})$ together with  a generalized complex structure $\J$ on $M$ such that $\J$ lifts to a generalized complex structure, $\J_{\elli}$ on $\elli$. We then transport $\J_{\elli}$ to a generalized complex structure $\hat{\J}_{\elli}$ on $\ellihat$ via $T$-duality and study when $\hat{\J}_{\elli}$ is itself a lift of a generalized complex structure $\hat{\J}$ on $\hat{M}$. When this happens, it makes sense to call $\hat{\J}$ the structure $T$-dual  to $\J$.  We also notice that the condition that a structure can be lifted to $\elli$ is a very restrictive one and therefore examples satisfying all the hypotheses above are rare. Yet, once one example is found, we can use the theory of blow-ups within generalized complex geometry \cite{BCD16} to produce infinite families of $T$-dual structures satisfying the lifting conditions described above. 

Stable generalized complex structures can be viewed as symplectic structures on $\elli$, and can hence be transported in this manner. There is however an intriguing difference with $T$-duality for principal torus actions: A generalized complex structure on $M$ which lifts to $\elli$ will always induce a generalized complex structure on the $T$-dual Lie algebroid $\ellihat \rightarrow \hat{M}$, however it might not induce a generalized complex structure on $M$. We will, for instance, encounter complex structures on the elliptic tangent bundle on manifolds which are not almost complex.

\subsection{Organisation of the paper}
This paper is organized as follows. In Section \ref{sec:backgroud} we will recall the notions of elliptic divisors, and their associated Lie algebroids. Furthermore, we recall how stable generalized complex structures can be viewed as symplectic structures on Lie algebroids. More generally, we discuss lifting generalized complex structures to Lie algebroids. In Section \ref{sec:classification} we recall the classification of torus actions by Heafliger-Salem. Afterwards we introduce the notion of a connection for a standard torus action, and show how these can be used to understand this classification. In Section \ref{sec:singgeometry} we study singular fibre products, which we will encounter in the description of $T$-duality. In Section \ref{sec:Tduality} we introduce the the notion of non-principal $T$-duality, using the Lie algebroids from Section \ref{sec:backgroud}. In Section \ref{sec:blowups} we show that the elliptic tangent bundle interacts very well with (generalized) complex blow-ups, and show that $T$-dualizing commutes with blowing-up. Finally, in Section \ref{sec:examples} we discuss concrete examples.

\subsection{Acknowledgements}

The second author was supported by the NWO through the UGC Graduate Programme (Netherlands), and through the FWO and FNRS under EOS project G0H4518N, and FWO Project G083118N (Belgium). The authors would like to thank Maarten Mol and \'Alvaro del Pino for useful discussions. The authors are grateful to the organizers of the workshop ``Generalized Geometry in Interaction'', Madrid 2022, were a preliminary version of this article was presented.

\section{Divisors and Lie algebroids; generalized complex structures and their lifts}\label{sec:backgroud}

As we mentioned in the introduction, the main objective of this paper is to extend $T$-duality to torus actions with fixed points without giving up on compactness of the spaces involved. The object that will allow us to accomplish this task are Lie algebroids, specifically Lie algebroids associated to divisors as introduced in \cite{CG17,CKW20}. Not only will these allow us to lift torus actions, but also we will arrive at the very natural problem of lifting geometric structures from the underlying manifold to the Lie algebroid.

In this section give the first step towards our goal and recall the theory of divisors and Lie algebroids (for more details, see \cite{CG17,CKW20}). We also introduce several relevant geometric structures, including generalized complex structures, and their lifts to these Lie algebroids. 

To fix notation, we recall:

\begin{definition}\label{def:lie algebroid}
A \textbf{Lie algebroid} $A \rightarrow M$ consists of a vector bundle $A$ over $M$, together with a Lie bracket $[\cdot,\cdot]$ on $\Gamma(\A)$, a vector bundle map $\rho : \A \rightarrow TM$ (called the \textbf{anchor}) such that the following Leibniz identity is satisfied:
\begin{equation*}
[v,fw] = f[v,w] + \mathcal{L}_{\rho(v)}(f)w, \quad \text{for all } v,w \in \Gamma(\A). \hfill \qedhere
\end{equation*} 
\end{definition}

\subsection{Divisors}
Inspired by divisors in algebraic geometry, we will make use of the following notion to deal with singularities of geometric structures:
\begin{definition}
A \textbf{real (resp. complex) divisor} on $M$ is a locally principal ideal $I \subset C^{\infty}(M;\rr)$ (respectively $C^{\infty}(M;\cc))$ which is locally generated by functions with a nowhere dense zero-set.
\end{definition}
\begin{definition}
Let $(M,I_M)$ and $(N,I_N)$ be manifolds together with divisors. A smooth map $\varphi : M \rightarrow N$ is a \textbf{morphism} of divisors if $\varphi^*I_N = I_M$. It is said to be a \textbf{diffeomorphism} of divisors if $\varphi$ is moreover a diffeomorphism.
\end{definition}

We will encounter three types of divisors in this article:
\begin{definition}\label{def:divisors}
Let $M^n$ be a smooth manifold.
\begin{itemize}
\item A \textbf{real log divisor} on $M$ is a real divisor, $I_Z$, such that for each $p \in M$ there are coordinates $(x_1,\ldots,x_n)\colon U \to \rr^n$ centred at $p$ and an integer $k$ such that
\begin{equation*}
I_Z|_U = \inp{x_1\cdot \ldots \cdot x_k} .
\end{equation*}
\item A \textbf{complex log divisor} on $M$ is a complex divisor $I_D$, such that for each $p \in M$ there are coordinates $(z_1,\ldots,z_k,x_{2k+1},\ldots,x_n)\colon U \to \cc^k \times \rr^{n-2k}$ centred at $p$ such that
\begin{equation*}
I_{D}|_U = \inp{z_1\cdot \ldots \cdot z_k}.
\end{equation*}
\item An \textbf{elliptic divisor} on $M$ is a real divisor $I_{\abs{D}}$, such that for each $p \in M$ there are coordinates $(x_1,\ldots,x_n)\colon U \to \rr^n$ centred at $p$ and an integer $k$ such that
\begin{equation*}
I_{\abs{D}}|_U = \inp{r_1^2\cdot \ldots \cdot r_k^2},
\end{equation*}
with $r_i^2 = x_{2i-1}^2 + x_{2i}^2$.
\end{itemize}
In these three cases, the number $k$ is the \textbf{multiplicity} of $p$. The vanishing set of the ideal is called the \textbf{vanishing locus}. We call each of these divisors \textbf{smooth} if their vanishing locus is an embedded smooth submanifold, that is, if the multiplicity of each point is either $0$ or $1$.
\end{definition}
In general, the vanishing locus is only (the image of) an immersed submanifold with transverse intersections, and is naturally stratified by the multiplicity of its points. It is convenient to introduce some notation to refer to the different strata of the vanishing locus, so if we let, say, $D$, denote the vanishing locus of any of the divisors above, the set of points with multiplicity at least $k$ will be denoted by $D(k)$ and  the set of points with multiplicity precisely $k$ will be denoted by $D[k]$. Each of the $D(k)$ is itself an immersed submanifold, and hence we may speak of:

%The vanishing locus $D$ is an immersed submanifold, and so are the sets $D(k)$. Hence we may define the following:
\begin{definition}
Let $I_{\abs{D}}$ be an elliptic divisor, and let $\iota : N \rightarrow D(k)$ be an immersion. The image of a connected component of $N$ is called an \textbf{irreducible component} of $D(k)$.
\end{definition}

Given a complex log divisor $I_D$, one may consider its \textbf{associated elliptic divisor} defined by the condition $I_{\abs{D}}\otimes \cc = I_D \otimes \bar{I}_D$. In this case, $D[1]$ is automatically co-oriented by the differential of any local generator of $I_D$. Conversely, given an elliptic divisor $I_{\abs{D}}$ for which $D[1]$ is co-oriented, there exists a unique complex log divisor $I_{D}$ inducing $I_{\abs{D}}$ and the given co-orientation.

\subsection{Lie algebroids and residue maps}
Using the Serre-Swann theorem, one may show that vector fields preserving the divisors from Definition \ref{def:divisors} define vector bundles and the natural identification of sections of these vector bundles with vector fields makes these vector bundles into Lie algebroids:
\begin{definition}
Let $M$ be a manifold, and $I_Z,I_D,I_{\abs{D}}$ a real log, complex log, respectively elliptic divisor on $M$. Vector fields preserving these ideals define Lie algebroids:
\begin{itemize}
\item $\A_Z \rightarrow TM$, the \textbf{real log-tangent bundle}.
\item $\A_D \rightarrow T_{\cc}M$, the \textbf{complex log-tangent bundle}.
\item $\elli \rightarrow TM$, the \textbf{elliptic tangent bundle}. \hfill \qedhere
\end{itemize}
\end{definition}
As the divisors under consideration have local normal forms, so do have the associated Lie algebroids:
\begin{remark}\label{rem:locform}
In the coordinates of Definition \ref{def:divisors}, we have:
\begin{itemize}
\item $\Gamma(\A_Z) = \inp{x_1\partial_{x_1},\ldots,x_k\partial_{x_k},\partial_{x_{k+1}},\ldots,\partial_{x_n}}$.
\item $\Gamma(\A_D) = \inp{z_1\partial_{z_1},\partial_{\bar{z}_1},\ldots,z_k\partial_{z_k},\partial_{\bar{z}_k},\partial_{x_{2k+1}},\ldots,\partial_{x_n}}$.
\item $\Gamma(\elli) = \inp{r_1\partial_{r_1},\partial_{\theta_1},\ldots,r_k\partial_{r_k},\partial_{\theta_k},\partial_{x_{2k+1}},\ldots,\partial_{x_n}}$,
\end{itemize}
with $r_i\partial_{r_i} = x_{2i-1}\partial_{x_{2i-1}}+x_{2i}\partial_{x_{2i}}$ and $\partial_{\theta_i} = x_{2i-1}\partial_{x_{2i}} - x_{2i}\partial_{x_{2i-1}}$.
\end{remark}

%{\color{red}
%\begin{remark}[Holomorphic log divisor]\label{rm:hologdiv}
%Let $M^n$ be a complex manifold, and $D$ a complex codimension one immersed submanifold, such that around each point there exist local holomorphic coordinates $z_1,\ldots,z_n$ such that $D = \set{z_1\cdot \ldots  \cdot z_k = 0}$. Then not only does $D$ define a complex divisor as in Definition \ref{def:divisors}, but also if we restrict ourselves to local holomorphic functions and vector fields we obtain the holomorphic analogue of a real log divisor. Further, the sheaf of holomorphic vector fields tangent to $D$, gives rise to the holomorphic log-tangent bundle, $T^{1,0}M(-\log D)$ with corresponding sheaf of differential forms $\Omega^{\bullet,0}(\log D)$.
%\end{remark}}

For each of the Lie algebroids above, we obtain a corresponding Lie algebroid de Rham complex on the exterior power of the dual vector bundle.
The relation between complex and elliptic tangent bundles translates into the existence of maps between their corresponding de Rham complexes. Indeed, given a complex log divisor,  $I_D$, and associated elliptic divisor, $I_{\abs{D}}$, taking the imaginary part of a complex log form is a map $\Im^* \colon (\Omega^{\bullet}(\A_D),d_{\A_D}) \rightarrow (\Omega^{\bullet}(\elli),d_{\elli})$.

There are several residue maps associated to these Lie algebroids, obtained by taking the coefficients of certain singular generators. We quickly recall the ones we need in this article.
Given a smooth elliptic divisor, with $D$ co-orientable, it is explained in \cite{CG17} that there are residue maps:
\begin{align*}
\Res_q : (\Omega^{\bullet}(\elli),d_{\elli}) \rightarrow (\Omega^{\bullet-2}(D),d), \qquad \alpha &\mapsto \iota_D^*(\iota_{r\partial_r\wedge\partial_{\theta}}\alpha)\\
\Res_r : (\Omega^{\bullet}(\elli),d_{\elli}) \rightarrow (\Omega^{\bullet-1}(S^1ND),d), \qquad \alpha & \mapsto \iota_{r\partial_r}\alpha,
\end{align*}
where for $\Res_r$, $S^1ND$ is the sphere bundle of the normal bundle of $D$ and the Lie algebroid 1-form dual to $\partial_\theta$ is interpreted as fiberwise volume form on the bundle $S^1ND \to D$.

Let $I_{\abs{D}}$ be a general elliptic divisor with $D[1]$ co-orientable. The inclusion $\iota_{M\backslash D(2)} : M\backslash D(2) \rightarrow M$, pulls the elliptic divisor $I_{\abs{D}}$ back to a smooth elliptic divisor with vanishing locus $D[1]$. Hence we may define residue maps by composition:
\begin{align*}
\Res_q : \Omega^{\bullet}(\elli) \rightarrow \Omega^{\bullet-2}(D[1]), \qquad \alpha &\mapsto \iota_{D[1]}^*(\iota_{r\partial_r\wedge\partial_{\theta}}(\iota_{M\backslash D(2)}^*\alpha))\\
\Res_r : \Omega^{\bullet}(\elli) \rightarrow \Omega^{\bullet-1}(S^1ND[1]), \qquad \alpha & \mapsto \iota_{r\partial_r}(\iota^*_{M\backslash D(2)}\alpha),
\end{align*}
where $S^1ND[1]$ is the sphere bundle of the normal bundle of $D[1]$. There exists several higher order residue maps, however defining these very quickly becomes a combinatorial task. As we are mainly interested in two-forms, we will define:
\begin{definition}\label{def:higherellres}
Let $I_{\abs{D}}$ be an elliptic divisor, $\omega \in \Omega^2(\elli)$ and let $p \in D[2]$. Using the coordinates from Remark \ref{rem:locform} we define
\begin{equation*}
\Res_{r_ir_j}\omega(p) = \omega_p(r_i\partial_{r_i},r_j\partial_{r_j}), \quad \Res_{r_i\theta_j}\omega(p) = \omega_p(r_i\partial_{r_i},\partial_{\theta_j}), \quad \Res_{\theta_i\theta_j}(\omega)(p) = \omega_p(\partial_{\theta_i},\partial_{\theta_j}). \hfill \qedhere
\end{equation*}
\end{definition}
Note that these residues are only well-defined up to sign. Similarly, one may define residue maps for the real log tangent bundle. If $I_Z$ is a smooth real log divisor, we may define
\begin{equation*}
\Res : \Omega^{\bullet}(\A_Z) \rightarrow \Omega^{\bullet-1}(Z), \quad \alpha \mapsto \iota^*_Z(\iota_{x\partial_x}\alpha).
\end{equation*}
Again, for a general real log divisor we may define $\Res^1 : \Omega^{\bullet}(\A_Z) \rightarrow \Omega^{\bullet-1}(Z[1])$ by precomposing with the pullback $M\backslash D(2) \hookrightarrow M$. The higher residues are again somewhat involved to define. For two-forms we define:
\begin{equation}\label{eq:higherlogres}
\Res^2_{x_ix_j}(\omega)(p) = \omega_p(x_i\partial_{x_i},x_j\partial_{x_j}), \quad p \in Z[2].
\end{equation}

\subsection{Generalized complex structures}
In this section we recall the definition of generalized complex structures and focus our attention on the class of stable generalized complex structures called stable. We will recall that these structures are in correspondence with symplectic structures on the elliptic tangent bundle.

Given a manifold $M$, we form $\mathbb{T}M = TM \oplus T^*M$. This bundle is endowed with the natural symmetric pairing given by evaluation of forms on vectors.
\begin{definition}
A \textbf{generalized complex structure} on a manifold $M$, is a pair $(\mathbb{J},H)$, where  $H \in \Omega^3(M)$ is  a closed real 3-form and $\mathbb{J}$ is a complex structure on $\mathbb{T}M$ orthogonal with respect to the natural paring and whose $+i$-eigenbundle, $L \subset \mathbb{T}_\cc M$ is involutive with respect to the Courant--Dorfmann bracket:
\begin{equation*}
\Cour{X+\xi, Y + \eta}_H = [X,Y] + \mathcal{L}_X\eta - \iota_Yd\xi + \iota_X\iota_YH, \quad X+ \xi,Y + \eta \in \mathbb{T}M.\hfill\qedhere
\end{equation*}
\end{definition}

Given a generalized complex structure one may decompose it as a block matrix with respect to the splitting of $\mathbb{T}M$. The skew-symmetry of $\mathbb{J}$ ensures this decomposition is of the form
\begin{equation*}
\mathbb{J} = \begin{pmatrix}
J & \pi_{\mathbb{J}}^{\sharp} \\
\sigma^{\flat} & -J^*,
\end{pmatrix}, \qquad J \in  \text{End}(TM), \sigma \in \Omega^2(M), \pi_{\mathbb{J}} \in \mathfrak{X}^2(M).
\end{equation*}
The bivector $\pi_{\mathbb{J}}$ is Poisson. Two generalized complex structures $(\mathbb{J}_1,H_1),(\mathbb{J}_2,H_2)$ are said to be \textbf{gauge equivalent}, if there exists a two-form $B \in \Omega^2(M)$ such that $dB = H_1-H_2$ and
\begin{equation*}
\mathbb{J}_2 = \begin{pmatrix}
1 & 0 \\
B^{\flat} & 1
\end{pmatrix}
\mathbb{J}_1
\begin{pmatrix}
1 & 0 \\
-B^{\flat} & 1
\end{pmatrix}.
\end{equation*}
If two generalized complex structures $\mathbb{J}_1,\mathbb{J}_2$ are gauge equivalent, then $\pi_{\mathbb{J}_1} = \pi_{\mathbb{J}_2}$.

The Clifford action of $X + \xi \in \mathbb{T}_{\cc}M$ on $\rho \in \wedge^nT_{\cc}M$ is defined via $(X+\xi) \cdot \rho = \iota_X \rho + \xi \wedge \rho$. Given a generalized complex structure, one may define its \textbf{canonical bundle} $K \subset \wedge^{\mathrm{top}}T_{\cc}M$ via
\begin{equation*}
L = \set{u \in \mathbb{T}_{\cc}M : u\cdot K = 0}.
\end{equation*}
The dual bundle $K^*$ has a natural section $s$, called the \textbf{anticanonical} section which is defined via $s : \Gamma(K) \rightarrow C^{\infty}(M;\cc), \rho \mapsto \rho_0$.
\begin{definition}
A generalized complex structure $(\mathbb{J},H)$ on $M$ is called \textbf{stable} if $s(\Gamma(K)) \subset C^{\infty}(M;\cc)$ defines a complex log divisor.
\end{definition}

\begin{example}
Let $(M,J,\pi)$ be a holomorphic Poisson manifold with $\pi = \pi_R + i \pi_I$ and denote by $\pi_I^\sharp\colon T^*M \to TM$ the map associated to $\pi_I$. Then
\[ \mathbb{J}_{J,\pi} :=
\begin{pmatrix}
-J &4 \pi_I^{\sharp} \\
0 & J^*
\end{pmatrix}
\]
defines a generalized complex structure. If furthermore $\wedge^n\pi(\Gamma(\wedge^{\mathrm{top}}(T^*M)) \subset C^{\infty}(M;\cc)$ defines a complex log divisor, $\mathbb{J}_{J,\pi}$ defines a stable generalized complex structure. Holomorphic Poisson structures with these properties are called \emph{log symplectic}.
\end{example}

These generalized complex structures can be described as symplectic structures on Lie algebroids:
\begin{theorem}[\cite{CKW20}]
There is a one-to-one correspondence between gauge equivalence classes $[(\mathbb{J},H)]$ of stable generalized complex structures on $M$ and symplectic forms $\omega \in \Omega^2(\elli)$ on the elliptic tangent bundle with $D[1]$ co-orientable and satisfying the following equations:
\begin{itemize}
\item $\Res^q(\omega) = 0$
\item $\Res_{r_i\theta_j}\omega = \Res_{\theta_ir_j}\omega$
\item $\Res_{r_ir_j}\omega = -\Res_{\theta_i\theta_j}\omega$
\end{itemize}
Given $(\mathbb{J},H)$, the associated $\omega \in \Omega^2(\elli)$ is given by $\pi_{\mathbb{J}}^{-1}$.
\end{theorem}

\subsection{Lifting generalized complex structures}
Lifting Poisson structures to Lie algebroids is a well-established field of study. However, lifts of generalized complex and Dirac structures to Lie algebroids have not yet received much attention. As we will see, the questions of which structures lift to Lie algebroids and which structures descend from Lie algebroids arise naturally.

\subsubsection{Severa class}
Let us first shortly study Courant algebroids where we replace the role of the tangent bundle by a general Lie algebroid.
Given any Lie algebroid $\A$ and $H \in \Omega^3(\A)$ closed one may define the standard $\A$-Courant algebroid, in the usual manner. We set $\mathbb{A} = \A\oplus \A^*$ and for $X+\xi,Y+\eta \in \Gamma(\A\oplus \A^*)$.
\begin{align*}
\inp{X+\xi,Y+ \eta} &= \frac{1}{2}(\eta(X) + \xi(Y)),\\
 \Cour{X+\xi,Y+\eta}_H &= [X,Y] + \mathcal{L}_X\eta - \iota_Yd\xi +\iota_X\iota_YH.
\end{align*}
This notion, which is a particular case of a Lie bialgebroid, can be axiomatised as follows:
\begin{definition}
An \textbf{exact $\A$-Courant algebroid} is a vector bundle extension
\begin{equation}\label{eq:Courant}
0 \rightarrow \A^* \overset{\pi^*}{\rightarrow} \mathbb{A} \overset{\pi}{\rightarrow} \A \rightarrow 0,
\end{equation}
together with a non-degenerate pairing $\inp{\cdot,\cdot}$ on $\mathbb{A}$, and a bracket $\Cour{\cdot,\cdot}$ on $\Gamma(\mathbb{A})$ satisfying:
\begin{itemize}
\item $\Cour{e_1,\Cour{e_2,e_3}} = \Cour{\Cour{e_1,e_2},e_3} + \Cour{e_2,\Cour{e_1,e_3}}$
\item $\pi(\Cour{e_1,e_2}) = [\pi(e_1),\pi(e_2)]_{\A}$
\item $\Cour{e_1,fe_2} = f\Cour{e_1,e_2} + \mathcal{L}_{\rho(\pi(e_1))}(f)e_2$
\item $\mathcal{L}_{\rho(\pi(e_1))}(\inp{e_2,e_3}) = \inp{\Cour{e_1,e_2},e_3} + \inp{e_2,\Cour{e_1,e_3}}$
\item $\Cour{e_1,e_1} = \pi^*d_{\A}\inp{e_1,e_1}$. \hfill \qedhere
\end{itemize}
\end{definition}
Given an exact $\A$-Courant algebroid, the composition $\rho \comp \pi : \mathbb{A} \rightarrow TM$ together with $\inp{\cdot,\cdot}$ and $\Cour{\cdot,\cdot}$ will give $\mathbb{A}$ the structure of a non-exact Courant algebroid. 

Just as for ordinary exact Courant algebroids, isotropic splittings $s : \A \rightarrow \mathbb{A}$ of \eqref{eq:Courant} are classified by closed three-forms $H \in \Omega^3(\A)$ appearing as the curvature
\begin{equation*}
\iota_X\iota_YH = s^*[s(X),s(Y)].
\end{equation*} 
The difference between two splittings $s-s'$, defines a two-form $B \in \Omega^2(\A)$ via $(s-s')(X) = \iota_XB$. Consequently $[H] \in H^3(\A)$ is independent of the isotropic splitting and determines the $\A$-Courant algebroid structure completely. Hence, this cohomology class is the natural extension of the \textbf{Severa-class} of exact Courant algebroids.

\subsubsection{Lifting Dirac structures}

The answer to the question on how to lift structures takes different forms depending on what one is trying to lift. Here we will focus on Dirac and generalized complex structures, whose definitions on $\mathbb{A}$ is completely analogous to the definition on $\mathbb{T}M$: 

\begin{definition}
Let $\A \rightarrow M$ be a Lie algebroid. An \textbf{$\A$-Dirac structure} is a subbundle $E_{\A} \subset \mathbb{A}$ together with a three-form $H \in \Omega^3_{\rm cl}(\A)$, such that $E_{\A}$ is isotropic with respect to the pairing and involutive with respect to the bracket on $\mathbb{A}$.

An \textbf{$\A$-generalized complex structure} is a pair $(\J_{\A},H_\A)$ where $H_\A \in \Omega^3_{\rm cl}(\A)$ and $\J_{\A}$ is an orthogonal  complex structure $\J_{\A}$ on $\mathbb{A}$ for which  $\mathbb{A}^{1,0}$, the $+i$-eigenbundle, defines an $\mathbb{A}_{\cc}$-Dirac structure.
\end{definition}

%{\color{red}First, let us recall the definition for Poisson structures:
%\begin{definition}
%Let $\A \rightarrow M$ be a Lie algebroid, and $\pi \in \mathfrak{X}^2(M)$ a Poisson structure. We say that $\pi_{\A} \in \mathfrak{X}^2(\A)$ is a \textbf{lift} of $\pi$, if $\pi_{\A}$ is $\A$-Poisson (i.e. $[\pi_{\A},\pi_{\A}]_{\A} = 0$) and $\rho_*(\pi_{\A}) = \pi$.
%\end{definition}
%Note that when $\A$ has dense isomorphism locus (as will always be the case in this article), the Poisson condition for $\pi_{\A}$ will follow directly from the Poisson condition for $\pi$. 
%
%
%The following is a natural extension of lifting Poisson structures to Dirac structures:}

\begin{definition}\label{def:Diraclift}
Let $\A \rightarrow M$ be a Lie algebroid. An $\A$-Dirac structure $(E_{\A},H_{\A})$ is a \textbf{lift} of a Dirac structure $(E,H)$ on $M$   if $H_{\A} = \rho^*H$ and $\mathfrak{F}_{\rho}(E_{\A}) = E$.
\end{definition}
Here $\mathfrak{F}_{\rho}$ denotes the Dirac pushforward via the anchor map.

This notion matches the notion of a lift of a Poisson structure, when that is considered as a Dirac structure. Indeed, given   a Lie algebroid $\A \rightarrow M$, a Poisson structure $\pi_{\A}\in \Gamma(\wedge^2\A)$ (that is $[\pi_{\A},\pi_{\A}] =0$) and a Poisson structure  $\pi \in \mathfrak{X}^2(M)$ one can readily see that  $\rho_*\pi_{\A} = \pi$ if and only if $E_{\pi_{\A}} = {\rm graph}(\pi_{\A}) \subset \mathbb{A}$ is a lift of $E_{\pi} = {\rm graph}(\pi)$.

For generalized complex structures, however, working only with the push-forward map is too restrictive, so the notion of lift requires some tweaking.

\begin{definition}Let $\rho\colon \A \to TM$ be a Lie algebroid over $M$\begin{itemize}
\item $\tilde X+ \tilde \xi \in \mathbb{A}$ is \textbf{related} to $X + \xi\in \mathbb{T}M$ if $\rho(\tilde X) = X$ and $\rho^*\xi = \tilde\xi$.
\item A generalized complex structure  $(\J_\A,H_\A)$ on $\mathbb{A}$ is a \textbf{lift} of a generalized complex structure $(\J,H)$ on $M$ if whenever $\tilde X+ \tilde \xi\in \mathbb{A}$ is related to $X + \xi\in \mathbb{T}M$, $\J_\A (\tilde X+ \tilde \xi)$ is related to $\J (X + \xi)$.\qedhere
\end{itemize}
\end{definition}

%
%
%\begin{lemma}
%Let $J : \A \rightarrow \A$ be an $\A$-complex structure, and let $E_{J} \subset \mathcal{A}_{\cc}$ denote the corresponding complex Dirac structure. Then $\mathfrak{F}_{\rho}(E) \cap T_{\cc}M = \rho(\A^{1,0})$. Hence $\mathfrak{F}_{\rho}(E)$ will only define a Dirac structure when $\rho$ has constant rank.
%\end{lemma}
%This is quite undesirable, as the following example shows:
%\begin{example}
%Let $I_{\abs{D}} = \inp{x^2+y^2}$ be the standard elliptic divisor on $\rr^2$. The corresponding Lie algebroid is $\elli = \inp{r\partial_r, \partial_{\theta}}$. There is a natural complex structure on $\A$, defined by $\tilde{J}(r\partial_r) = \partial_{\theta}$. In complex coordinates this would read $\tilde{J}(z\partial_z) = iz\partial_z$. The Dirac pushforward of $L_{\tilde{J}}$ is given by $\mathfrak{F}_{\rho}(L_{\tilde{J}}) = \inp{z\partial_z, d\bar{z}}$, which is not smooth. Nevertheless, for the standard complex structure $J$ on $\cc$, we see that $J \comp \rho = \rho \comp \tilde{J}$.
%\end{example}
%The previous example shows that Definition \ref{def:Diraclift} is unsatisfactory for complex structures, hence we say:
\begin{example}\label{def:complexlift}
In the situation above, if  $J \colon TM \rightarrow TM$  and $J_{\A} \colon \A \to \A$ are complex structures we can promote them to generalized complex structures and ask when $J_{\A}$ is a lift of $J$. It turns out this happens if and only if the following diagram commutes
\begin{center}
\begin{tikzcd}
\A \ar[r,"J_{\A}"] \ar[d,"\rho"] & \A \ar[d,"\rho"] \\
TM \ar[r,"J"] & TM
\end{tikzcd}
\end{center}
which agrees with what one would like to call the lift of a complex structure. 
\end{example}

%\begin{remark}
%If $\pi_{\A}$ is $\A$-Poisson, then $\rho(\pi_{\A})$ is a Poisson structure. This is not the case for an $\A$-complex structure $J_{\A}$. Rather, if $\A$ has almost injective anchor the opposite is true:  $J$ determines $J_{\A}$, rather than the other way around. With this in mind, the terminology lift is a bit misleading in the setting of complex structures

%{\color{red} This was misleading. If $\A$ has almost surjective anchor, $J_\A$ determines $J$. IF almost injective, the other way around. Only saying half of the statement is confusing..}
%\end{remark}

%As we concluded that Definition \ref{def:Diraclift} is undesirable for generalized complex structures, we give an alternative definition. Given a Lie algebroid $\A \rightarrow M$, let $\rho^* : T^*M \rightarrow \A^*$ denote the dual of the anchor map. We say that $X + \xi \in \A \oplus \A^*$ is \textbf{$\rho$-related} to $Y + \eta \in TM \oplus T^*M$ if $\rho(X) = Y$ and $\rho^*\eta = \xi$. We denote this by $X+ \xi \simeq_{\rho} Y + \eta$
%\begin{definition}
%An $\A$-generalized complex structure $\mathbb{J}_{\A}$ is a \textbf{lift} of a generalized complex structure $\mathbb{J}$ on $M$ if:
%\begin{equation*}
%X + \xi \simeq_{\rho} Y + \eta \Leftrightarrow \mathbb{J}_{\A}(X+\xi) \simeq_{\rho} \mathbb{J}(Y+\eta),
%\end{equation*}
%for all $X+\xi \in \A \oplus \A^*$ and $Y + \eta \in TM \oplus T^*M$.
%\end{definition}
%Note that the implication from left to right, implies the other implication. It is immediate to check that this definition coincides with Definition \ref{def:complexlift} for complex structures. 
%
\begin{remark}
In this paper we will work with Lie algebroids $\A$ for which the locus $U$ where $\rho\colon \A \to TM$ is an isomorphism of vector bundles is dense in $M$. In this case, we see that $\J_{\A}$ is a lift of $\J$ on $U$ if and only if $\J_\A|_U= (\rho|_U) \comp \J\comp (\rho|_U)^{-1}$. If that holds, by continuity $\J_\A$ is related to $\J$ in all of $M$. Therefore we see that the question of whether $\J$ admits a lift to a Lie algebroid with dense isomorphism locus becomes the question of existence of a smooth extension of $(\rho|_U) \comp \J\comp (\rho|_U)^{-1}$ beyond the isomorphism locus.
\end{remark}
 
\begin{example}
It is also good to keep in mind that there are (generalized) complex structures on Lie algebroids, which do not arise as lifts of structures on $M$.
Consider, for example, the elliptic tangent bundle $\elli$ on $\cc^2$, associated to the ideal $\inp{r_1^2r_2^2}$. The following defines a complex structure on $\elli$:
\begin{align*}
z_1\partial_{z_1} \mapsto z_2\partial_{z_2},& \quad \bar{z_1}\partial_{\bar{z_1}} \mapsto \bar{z_2}\partial_{\bar{z_2}},\\
z_2\partial_{z_2} \mapsto -z_1\partial_{z_1},& \quad \bar{z_2}\partial_{\bar{z_2}} \mapsto -\bar{z_1}\partial_{\bar{z_1}}.
\end{align*}  
However, the induced complex structure on $\cc^2\setminus D$ does not extend over $D$, as we have, for example,
\begin{equation*}
\partial_{z_1} \mapsto \frac{z_2}{z_1} \partial_{z_2}. \hfill \qedhere
\end{equation*}
\end{example}

\section{Classification of torus actions and connections}\label{sec:classification}
Having introduced the geometric structures of interest last section, now we move on to study manifolds with the relevant torus actions. We begin this section by recalling the classification of certain torus actions by Haefliger and Salem in \cite{HS91}. The relevant cohomology class for this classification is $\check{H}^2(B;\zz^n)$, where $B$ is the quotient space of the torus action. Although the Haefliger--Salem class greatly resembles the Chern class of a principal $S^1$-bundle, no such description exists at the moment in terms of connections and curvature. We will use the elliptic tangent bundle to provide a connection point of view on this class.
\subsection{Haefliger--Salem classification}
The first ingredient of the classification is the following result:
\begin{theorem}[Theorem 3.1 in \cite{HS91}]\label{th:HS1}
Let $M$ be a manifold with a $T^k$-action, and let $\pi \colon M \rightarrow B$ denote the quotient map.
Any diffeomorphism $h\colon M \to M$ commuting with the action of $T^k$ and preserving all orbits is of the form $h(x) = f(\pi(x))\cdot x$, where $f \in C^{\infty}(B;T^k)$.
\end{theorem}
Here we say that $f \in C^{\infty}(B)$ if $\pi^*f$ is smooth on $M$. Similarly, we define $\Omega^{\bullet}(B)$ to correspond to the basic forms on $M$:
 \[\Omega^{\bullet}_{\rm bas}(M)=\{\xi \in \Omega^\bullet(M)\colon \iota_X \xi = \mathcal{L}_X \xi =0, \mbox{ for all infinitesimal generators } X \} \]
\begin{remark}
The discussion in \cite{HS91} concerns itself with the more general setting in which $M$ is also allowed to be an orbifold. As our case of interest is when the total space is a smooth manifold we will immediately restrict to that case.
\end{remark}

\begin{definition}
Let $M,M'$ be two manifolds endowed with $T^k$-actions with common orbit space $B$ and let $\pi : M \rightarrow B,\pi' : M' \rightarrow B$ denote the quotient maps. Then $M$ and $M'$ are said to be \textbf{locally equivalent} if each point in $B$ has a neighbourhood $U$ such that there exists a $T^k$-equivariant diffeomorphism from $\pi^{-1}(U)$ to $\pi'^{-1}(U)$, covering the identity on $B$.
\end{definition}

%\begin{remark}
%The notion of locally equivalent actions is somewhat subtle. All standard $T^k$-actions over the same base have the same normal form. However, the local coordinates in this normal form do not only use a equivariant diffeomorphism, there is also a conjugation of $T^k$. Therefore, not all standard $T^k$ actions are locally equivalent, and in Section \ref{sec:topchange} we will see some explicit examples where they are not.
%\end{remark}

When fixing an effective $T^k$-action on a manifold $M$ it is possible to classify all other such actions which are locally equivalent to $M$:
\begin{proposition}[Proposition 4.2 in \cite{HS91}]\label{prop:classification}
Let $\pi : M \rightarrow B$ be a manifold endowed with an effective torus action. The equivalence classes of manifolds endowed with torus actions locally equivalent to that on $M$ form a torsor for $\check{H}^2(B;\mathbb{Z}^k)$.\end{proposition}
Since we will need to get our hands on the cohomology classes prescribed by the proposition, we give a sketch of the proof. Notice that there is no preferred manifold endowed with a torus action locally equivalent to that of $M$ and the class in question measures how different from $M$ another such manifold is.
\begin{proof} 
Let $\pi' : M' \rightarrow B$ be a manifold endowed with a locally equivalent torus action. Therefore, there exists a cover $\mathcal{U} = \set{U_i}_{i\in I}$ of $B$ together with equivariant diffeomorphisms $h_i : \pi^{-1}(U_i) \rightarrow \pi'^{-1}(U_i)$. The transition maps $h_j^{-1} \comp h_i : \pi^{-1}(U_i \cap U_j) \rightarrow \pi^{-1}(U_i \cap U_j)$ satisfy the assumptions of Theorem \ref{th:HS1} and therefore there exists $f_{ji} \in C^{\infty}(U_i \cap U_j;T^k)$ such that $h_j^{-1}\comp h_i(x) = f_{ji}(\pi(x))\cdot x$. These functions combine to give a cocycle $\set{f_{ji}} \in \check{C}^1(\mathcal{U};\mathcal{T}^k)$, where $\mathcal{T}^k$ is the sheaf of smooth functions with values in $T^k$. Using the exponential sequence one obtains the desired class in $\check{H}^2(\mathcal{U};\mathbb{Z}^k)$.

Conversely, given a class in $\check{H}^2(\mathcal{U};\mathbb{Z}^k)$, one uses again the exponential sequence to obtain a cocyle $\set{f_{ji}} \in \check{C}^1(\mathcal{U};\mathcal{T}^k)$. Using this cocycle one may glue the disjoint unions of the $\pi^{-1}(U_i)$ to obtain a manifold $M'$ endowed with a locally equivalent torus action. One readily checks that cohomologous cocycles result in equivalent $T^k$ manifolds.
\end{proof}
\begin{definition}
Let $\pi : M \rightarrow B$ be a manifold endowed with an effective torus action with connected isotropy groups, and $\pi' : M' \rightarrow B$ be a locally equivalent action. We will refer to the corresponding class in $\check{H}^2(B,\mathbb{Z}^k)$ as the \textbf{Haefliger-Salem class} of $M'$ with respect to $M$.
\end{definition}

\begin{remark}
A key difference with the classification of principal torus bundles in terms of their Chern class, is that the classification in Proposition \ref{prop:classification} is affine in nature. Given a base $B$, there is no canonical torus actions playing the role $B \times T^k$ has in the principal case. 
\end{remark}

\begin{remark}
Note that the results in \cite{HS91} are about orbifolds, but as we assume that $M$ is a smooth manifold, any locally equivalent spaces will be smooth manifolds as well.
\end{remark}

%The next example provides a family of non-equivalent torus actions on the same manifold:
%\begin{example}[$S^2\times S^2$]
%We consider $S^2\times S^2$ by glueing two copies of $S^2\times D^2$ along their common boundary $S^2\times S^1$. Rather then glueing with the identity, one may glue via
%\begin{equation*}
%\varphi :S^2\times S^1 \rightarrow S^2 \times S^1, (z,t,w) \mapsto (w^nz,t,w), 
%\end{equation*}
%where we view $S^2 \subset \cc^2 \times \rr$, and $\S^1 \subset \cc$. The resulting space will be (diffeomorphic to) $S^2\times S^2$ independent of the value of $n$, and the resulting $S^1$-actions will be locally equivalent by construction. However, the actions are not globally equivalent, which we will now illustrate by explicitly computing the Haefliger Salem class. \alert{I'm confused.}
%\end{example}

\subsection{Torus actions, divisors and connections}
We will lift torus actions to the elliptic tangent bundle in such a way that the associated infinitesimal action becomes free. The required notion we need is the following:

\begin{definition}
A $T^k$ action and an elliptic divisor, $I_{\abs{D}}$ on a manifold $M$ are \textbf{compatible} if
\begin{itemize}
\item the action is effective,
\item the isotropy group of each point is connected,
\item the points with non-trivial isotropy are contained in $D$, 
\item the action preserves the ideal $I_{\abs{D}}$.\qedhere
\end{itemize}
\end{definition}
These assumptions allow us to lift to the elliptic tangent bundle:

\begin{proposition}\label{prop:inflift}
Let $I_{\abs{D}}$ be an elliptic divisor compatible with a $T^k$-action on $M$. Then there exists an injective Lie algebroid homomorphism $\tilde{a} : \mathfrak{t}^k \rightarrow \elli$ such that the following diagram commutes:
\begin{center}
\begin{tikzcd}
 & \elli \ar[d] \\
\mathfrak{t}^k \ar[ru,"\tilde{a}"] \ar[r,"a"] & TM
\end{tikzcd}
\end{center}
where $a\colon \mathfrak{t}^k \to TM$ is the infinitesimal $T^k$-action on $M$.
\end{proposition}

\begin{proof}
The existence of the Lie algebroid map is a local statement, so we can work in neighbourhoods of points.

Given a point $p$ where action is free, there are two possibilities, either $p\not\in D$ or $p\in D$. In the first case, $\A$ isomorphic to $TM$ via the anchor map in a neighbourhood of $p$ and the action has a unique lift to $\elli$. If $p \in D$, because the action preserves the ideal $\elli$, then the orbit of $p$ lies within the same stratum as $p$, say $D[i]$ and because the action is free, we get a rank $k$ subbundle $a(\mathfrak{t}^k) \subset TD[i] \hookrightarrow \A_{|D|}$, and therefore $a$ has a lift, which is unique by continuity.

Given a point where the action has nontrivial isotropy, say $T^l\subset T^k$, we can split $T^k = T^l \times T^{k-l}$. By the previous argument the action of $T^{k-l}$ lifts to $\A_{|D|}$ and the lift is made of a subspace that maps injectively to $TM$ via the anchor. Now we focus on the action of $T^l$. Since $p$ is fixed by the action of $T^l$ and the action is effective we can linearise it and assume that, in appropriate coordinates (with values in $\cc^m \times \rr^{n-2m}$), the action is given by
\[(\theta_1,\dots, \theta_l)\cdot (z_1,\dots, z_m,x_1,\dots, x_{n-2m}) \mapsto (e^{i(\sum_j a_{1,j}\theta_j)}z_1,\dots, e^{i(\sum_j a_{m,j}\theta_j)}z_m,x_1,\dots,x_{n-2m}),  \]
where $m \geq l$ and the matrix $A = (a_{i,j})_{ij}$ has integral coefficients. Because all isotropy groups are connected $A$ has an $l\times l$ submatrix with determinant $\pm 1$.

Suppose $f$ is a local generator of the elliptic divisor $I_{\abs{D}}$. Since $D$ has codimension 2, the locus where $f$ is nonzero is connected and hence we may take $f$ to be non-negative, vanishing only at $D$. Since $f$ is non-negative, by taking its average over the $T^l$ action, we obtain another nonnegative generator of the ideal, still denoted by $f$, which is invariant under the $T^l$ action. The invariant polynomials of the $T^l$-action are generated by $(r_1^2,\ldots,r_m^2,x_1,\ldots,x_{n-2m})$, and hence a result by Schwarz \cite{MR370643} implies that $f$ must be of the form $f = g(r_1^2,\ldots,r_m^2,x_1,\ldots,x_{n-2m})$, for some smooth $g$. Because $f$ has this form, it is immediate that the vector fields $\partial_{\vartheta_1},\ldots,\partial_{\vartheta_m}$  preserve the ideal generated by $f$, where we use polar coordinates  for $z_i$: $z_i = r_i e^{i\vartheta_i}$. Hence $\partial_{\vartheta_i}$ give rise to (non-vanishing) elements of $\elli$. Because the vector fields generating the $T^k$-action are given by $A(\partial_{\theta_1}),\ldots,A(\partial_{\theta_l}) \in \mbox{span}\langle\partial_{\vartheta_1},\ldots,\partial_{\vartheta_m}\rangle$, and $A$ has maximal rank, we see that these vector fields span a rank $l$-subspace of $\elli$ which over $D$ are in the kernel of the anchor map and hence are independent from the generators corresponding to the $T^{k-l}$ action. %This finishes the proof that the lift $\tilde{a}$ is injective. \alert{Can do with some cleaning.}
\end{proof}

Studying general torus actions goes beyond the scope of this article. For instance, one problem which occurs is that there is too much freedom in choosing a compatible elliptic divisor:
\begin{example}
Consider $\cc^2$, endowed with the $S^1$-action $\theta_1\cdot(z_1,z_2) = (e^{i\theta_1}z_1,z_2)$ and the compatible elliptic divisor generated by $\abs{z_1}^2\abs{z_2}^2$. Using coordinates $(t,z_2)$ for the quotient space $B = \rr_{\geq 0} \times \cc$, we have that $B$ inherits a quotient divisor $I_{\red} = \inp{t\abs{z_2}^2}$. Which is a mixture between an elliptic and real log divisor. More generally, for a $T^k$-action on $\cc^n$ the induced divisor on the base can have any kind of mixture between elliptic and real log divisors.  
\end{example}

We restrict ourselves to a very well-behaved class of singular torus actions, but still interesting enough to allow for many examples:
\begin{definition}\label{def:standard action}
We say that a $T^k$-action on a manifold $M^{n}$ is \textbf{standard} if each orbit has a neighbourhood $U$ such that
\begin{equation*}
U \simeq  T^{k-l} \times \cc^l \times \rr^m \qquad\mbox{ for some } k,l,m \mbox{ such that } k+l+m=n,
\end{equation*}
in which the action is given by
\begin{align*}
(e^{i\theta_1},\ldots,e^{i\theta_k})&\cdot (\varphi_1,\ldots,\varphi_{k-l}, z_1,\ldots,z_l,x_1,\ldots, x_m) \\&= (e^{i\theta_1}\varphi_1,\ldots,e^{i\theta_{k-l}}\varphi_{k-l},e^{i\theta_{k-l+1}}z_1,\ldots,e^{i\theta_k}z_l,x_1,\ldots, x_m),
\end{align*}
possibly after changing the generators of $T^k$.
\end{definition}
For general $T^k$-actions, deriving conditions when they are standard is not simple. However, for full torus actions one can show the following:
\begin{lemma}[\cite{CKW20B}, Proposition 6.1]
A $T^n$-action on a manifold $M^{2n}$ which is effective and has connected isotropy groups is standard.
\end{lemma}

The local normal form of a standard $T^k$-action immediately implies the following:

\begin{proposition}\label{prop:quotientmap}
Let $\pi: M \rightarrow B:= M/T^k$ be a manifold endowed with a standard $T^k$-action. Then
\begin{itemize}
\item $B$ is a manifold with corners.
\item The action is compatible with the elliptic divisor $I_{\abs{D}} : = \pi^*I_{\partial B}$.
\item $\pi$ induces a Lie algebroid submersion $\elli \rightarrow \A_{\partial B}$.
%\item the intersection stratification of the elliptic ideal coincides with the stratification by orbit types on $M$;
%\item $ND[1]$ is co-orientable;
%\item if $M$ is oriented, then so is $N$;
%\item if $M$ is four-dimensional and the action is not free, $f$ is homologically essential.
\end{itemize}
\end{proposition}
In particular there is a canonical compatible elliptic divisor associated to a standard $T^k$-action. Therefore, in what follows when given a standard $T^k$-action on $M$, $I_{\abs{D}}$ will always denote the canonical associated elliptic divisor.

\begin{example}
Consider $S^n \subset \rr^{n+1}$, and the circle action given by rotating the first two coordinates.
This action is standard with quotient space diffeomorphic to $D^{n-1}$ and elliptic divisor generated by $x_1^2+x_2^2$.
\end{example}

Thanks to Proposition \ref{prop:inflift} and \ref{prop:quotientmap} a standard $T^k$-action lifts to the associated Lie algebroid and so does the quotient map. Therefore, we obtain a diagram as follows:
\begin{center}
\begin{tikzcd}
\elli \ar[d] \ar[r,"\tilde{\pi}"] & \A_{\partial_B} \ar[d] \\
M \ar[r,"\pi"] & B
\end{tikzcd}
\end{center} 
Hence, we obtain a short exact sequence of vector bundles over $M$:
\begin{equation*}
0 \rightarrow \ker d\tilde{\pi} \rightarrow \elli \rightarrow \pi^*\A_{\partial_B} \rightarrow 0
\end{equation*}
We may quotient by $T^k$, to obtain the sequence of vector bundles over $B$:
\begin{equation*}
0 \rightarrow B \times \mathfrak{t}^k \rightarrow \elli/T^k \rightarrow \A_{\partial_B} \rightarrow 0
\end{equation*}
One may think of this sequence as an $\elli$ version of the Atiyah sequence for principal bundles. With this in mind, vector bundle splittings of this sequence give rise to the notion of connections for the $T^k$-action on $M$. Concretely, in terms of differential forms we have:
\begin{definition}
Let $M$ be a manifold endowed with a $T^k$ action, $m\colon T^k\times M \to M$, and let $I_{\abs{D}}$ be a compatible elliptic divisor. A \textbf{connection-one form} on $M$ is a Lie algebroid one-form $\Theta \in \Omega^1(\elli;\mathfrak{t}^k)$ such that
\begin{itemize}
\item $m_{\lambda}^*\Theta = \Theta$ for all $\lambda \in T^k$ (where $m_\lambda = m(\lambda,\cdot)$)
\item $\iota_{a(v)}\Theta = v$, for all $v \in \mathfrak{t}^k$. \hfill \qedhere
\end{itemize}
\end{definition}

\subsubsection{Complex line bundles}

Let $L \rightarrow M$ be a complex line bundle. Then $\Gamma(L^{1,0*}) \subset C^{\infty}(L;\cc)$ naturally defines a smooth complex log divisor. The associated elliptic divisor $I_{\abs{D}}$ is naturally compatible with the standard $S^1$-action on $L$.

Connections for the $S^1$-action arise from linear connections on $L$:
\begin{lemma}\label{lem:linebundleconn}
Let $\nabla$ be a complex linear connection on a complex line bundle $L \rightarrow M$, and let $\zeta \in \Omega^1(L\backslash M;\cc)$ denote the corresponding connection one-form on the associated $\cc^*$-bundle. Then $\zeta$ extends over the zero-section as a complex logarithmic form, and $\Theta := \Im(\zeta) \in \Omega^2(\elli)$ defines a connection one-form for the associated $S^1$-action
\end{lemma}
\begin{proof}
Let $z_{\alpha} : L|_{U_{\alpha}} \rightarrow \cc$ be local fibre coordinates of $L$. 
Write $z_{\alpha} = r_{\alpha}e^{i\theta_{\alpha}}$, and consider the complex log forms $d\log z_{\alpha} = d\log r_{\alpha} + i d\theta_{\alpha}$. If $z_{\alpha} = g^{\beta}_{\alpha}z_{\beta}$, with $g^{\alpha}_{\beta}$ transition functions, we find that $d\log z_{\alpha} - d\log z_{\beta} = d\log g^{\alpha}_{\beta}$. 
Given a complex linear connection $\nabla$ on $L$, we may consider the corresponding connection one-forms $\xi_{\alpha} \in \Omega^1(U_{\alpha};\cc)$, which satisfy $\xi_{\alpha}-\xi_{\beta} = d\log g^{\alpha}_{\beta}$. We conclude that $\zeta := d\log z_{\alpha} + \xi_{\alpha}$ defines a global complex logarithmic form
\end{proof}
If we let $K = K_1 + iK_2$ denote the curvature of the connection $\nabla$ on $L$, we find that $d\Theta = K_2$ is actually a smooth form.
\begin{remark}
The form $\Theta$, restricted to the complement of the zero-section, also goes by the name of the global angular form. The above discussion shows that the global angular form can be extended over the zero-section as an elliptic differential form.
\end{remark}

%Locally equivalent torus actions carry comparable compatible elliptic divisors:
%\begin{lemma}\label{lem:locequivdiv}
%Let $M$ be endowed with a $T^k$-action and a compatible elliptic divisor $I_{\abs{D}}$. Then any manifold $M'$ equipped with a locally equivalent action has a compatible elliptic divisor $I_{\abs{D'}}$ for which the local $T^k$-equivariant diffeomorphisms are morphisms of divisors.
%\end{lemma}
%\begin{proof}
%Define $I_{\abs{D'}} := \tilde{q}^*(I_{\abs{D}}/T^k)$. Note that if $h : q^{-1}(V) \rightarrow \tilde{q}^{-1}(V')$ is a local equivariant diffeomprhism, then $I_{\abs{D'}} = (h^{-1})^*I_{\abs{D}}$. Hence $I_{\abs{D'}}$ is an ideal of smooth functions, and in particular an elliptic divisor.
%\end{proof}

\subsection{Smoothness of the curvature}
Given a standard $T^k$-action and a corresponding connection one-form $\Theta \in \Omega^1(\elli;\mathfrak{t}^k)$, the curvature, $d\Theta$, is basic and descends to a form in $\Omega^2(\A_{\partial B};\mathfrak{t}^k)$. We saw for complex line bundles that there are connections for which the curvature was in fact a smooth form. This is a general phenomenon for standard torus actions. To prove this, we need to understand the cohomology of the log-tangent bundle on the base. Below we use the residue maps introduced in Definition \ref{def:higherellres} and the discussion thereafter.

\begin{proposition}[Theorem 29 in \cite{MS18},7.24 in \cite{GLPR17}]
Let $(M,I_Z)$ be a manifold endowed with real log divisor. Then there exists an isomorphism
\begin{equation}\label{eq:basecohom}
H^{\bullet}(\mathcal{A}_Z) \simeq H^{\bullet}(M) \oplus H^{\bullet-1}(Z[1]) \oplus H^{\bullet-2}(Z[2]) \oplus \cdots,
\end{equation}
where $Z[n]$ denotes the intersection stratification of $Z$. The maps $H^{\bullet}(\mathcal{A}_Z) \rightarrow H^{\bullet-i}(Z[i])$ are induced by the residue maps $\Res^i$, and the map $H^{\bullet}(M) \rightarrow H^{\bullet}(\mathcal{A}_Z)$ is the pullback via the anchor.
\end{proposition}

Using \eqref{eq:basecohom} we see that we can associate a well-defined smooth cohomology class to a class in $H^{\bullet}(\A_Z)$. Explicitly we may also phrase this as follows:
\begin{lemma}\label{lem:smoothandexact}
Let $(M,Z)$ be a manifold endowed with a immersed hypersurface, with self-transverse intersections. Let $\alpha \in \Omega^{\bullet}(M)$ be such that $[\rho^*\alpha] = 0 \in H^{\bullet}(\A_Z)$. Then $[\alpha] = 0 \in H^{\bullet}(M)$.
\end{lemma}

Recall from Lemma \ref{prop:quotientmap} that the quotient map of a standard $T^k$-action induces a Lie algebroid submersion $\pi : \elli \rightarrow \mathcal{A}_{\partial B}$. This map is compatible with the residues in the following manner. 
\begin{lemma}
Let $\pi : M \rightarrow B$ be endowed with a standard $T^k$-action. Then for $i = 1,2$ the following diagram commutes:
\begin{center}
\begin{tikzcd}
\Omega^{2}(\elli) \ar[r,"\Res_r^i"] & \Omega^{2-i}(T^iND[i]) \\
\Omega^2(\mathcal{A}_{\partial B}) \ar[r,"\Res^i"] \ar[u,"\pi^*"] & \Omega^{2-i}(\partial B[i]) \ar[u,"\pi^*"]
\end{tikzcd}
\end{center}
\end{lemma}
As $D[2]$ is a discrete set of points, $\Res^2_r\omega$ is defined to be $\Res_{r_1r_2}(\omega)(p)$. A similar result holds for higher order residues, but is more involved to state.

Using the description of the cohomlogy of $\A_{\partial B}$, we can show that the curvature of a connection, $d\Theta$, is cohomologically smooth:
\begin{lemma}\label{lem:smoothcurv}
Let $M$ be endowed with a standard $T^k$-action and let $\Theta \in \Omega^1(\elli;\mathfrak{t}^k)$ be a connection-one form with curvature $K \in \Omega^2(\mathcal{A}_{\partial B},\mathfrak{t}^k)$. Then there exists a smooth form $\tilde{K} \in \Omega^2(B,\mathfrak{t}^k)$ with $[\rho^*\tilde{K}] = [K] \in H^2(\mathcal{A}_{\partial_B};\mathfrak{t}^k)$.
\end{lemma}
\begin{proof}
Using \eqref{eq:basecohom} we have to show that the residues $\Res^i$ of $K$ vanish. By rank reasons $\Res^i$ vanishes for $i\geq 3$. We have that $p^*\Res^1(K) = \Res^1_r(d\Theta) = d \Res^1_r(\Theta)$. By $T^k$-invariance of $\Theta$ we must have that $\Res^1_r(\Theta)$ is $T^k$-invariant as well, and thus of the form $p^*f$, with $f \in C^{\infty}(B)$. Therefore $p^*\Res^1(K) = p^*df$, and we conclude that $\Res^1(K)$ is exact. We also have that $p^*\Res^2(K) = \Res^2_r(d\Theta)$, but the latter is necessarily zero by smoothness of the coefficient functions of elliptic one-forms. We conclude that all the residues of $[K]$ vanish, as these are precisely the maps $H^{\bullet}(\A_{\partial B}) \rightarrow H^{\bullet-i}(\partial B[i])$ in Equation \eqref{eq:basecohom}, we conclude that $[K] = [\rho^*\tilde{K}]$, for some smooth $\tilde{K} \in \Omega^2(B;\mathfrak{t}^k)$. 
\end{proof}

The following example shows that Lemma \ref{lem:smoothcurv} is not a given when the elliptic divisor and action are independent:
\begin{example}
Let $M$ be a manifold endowed with a smooth elliptic divisor $I_{\abs{D}}$, with $D$ co-orientable, and consider the trivial $S^1$-bundle, $\pi : M \times S^1 \rightarrow M$. The elliptic divisor is trivially compatible with the $S^1$-action, and hence we can choose a connection one-form for this action. Any such form will be given by $\Theta = \alpha + d\theta$, with $\alpha \in \Omega^1(\elli)$ and $d\theta \in \Omega^1(S^1)$.

We may pick $\alpha = 0$, and then the curvature of $\Theta$ simply vanishes. However, we may change $\alpha$ in such a manner that $d\Theta$ remains smooth but is no longer exact as an element in $H^2(M)$, contradicting the conclusion of Lemma \ref{lem:smoothcurv}.

The construction goes as follows: Consider a metric on a tubular neighbourhood $\mathcal{U}$ of $D$ in $M$, and let $r^2 : \mathcal{U} \rightarrow \rr$ denote the associated squared distance function. Let $\Theta_M \in \Omega^1(\elli)$ be a connection one-form on $\mathcal{U}$, obtained as in Lemma \ref{lem:linebundleconn} by viewing $\mathcal{U}$ as an (open) in a complex line bundle. 

Let $f : \rr \rightarrow \rr$ be a smooth function with ${\rm supp} f \subset [-1,1]$, and $f|_{[-\epsilon,\epsilon]} \equiv 1$ for some $\varepsilon < 1$. Define an elliptic form $\alpha = f(r^2)\Theta_M$. The cohomology class $d\alpha$ is concentrated in $\mathcal{U}$, where it is precisely the curvature of the normal bundle of $D$. Consequently, we see that $\Theta = \alpha + d\theta$ has smooth curvature, which is not cohomologous to zero.
\end{example}

\subsubsection{Local description of connection one-forms}
For future reference we recall the local description for connection one-forms on principal bundles:
\begin{proposition}\label{prop:localconform}
Let $\pi : P \rightarrow B$ be a principal $G$-bundle, and let $h_i : \pi^{-1}(U_i) \rightarrow U_i \times G$ be local trivialisations for $i=1,2$. Let $\Theta\in \Omega^1(P;\mathfrak{g})$ be a connection one-form on $P$. Then
\begin{equation*}
(h_i^{-1})*\Theta = \omega_{\rm MC} + \pi^*\alpha_i,
\end{equation*}
with $\omega_{\rm MC} \in \Omega^1(G;\mathfrak{g})$ the Maurer-Cartan form, and $\alpha_i \in \Omega^1(U_i;\mathfrak{g})$. Moreover, if we let $f_{ji} \in C^{\infty}(U_i\cap U_j;G)$ be the transition functions given by $h_j^{-1} \comp h_i(x,g) = (x,f_{ji}(x)\cdot g)$, we have
\begin{equation*}
\alpha_i - {\rm Ad}^*_{f_{ji}}\alpha_j = \pi^*d\log f_{ji}.
\end{equation*}
\end{proposition}

\subsection{Cech--de Rham isomorphism}
To continue, we need to recall a basic version of the \v{C}ech-de Rham isomorphism.

Let $M$ be endowed with a proper group action by some group $G$. Let $\Omega^n_{\bas}$ denote the sheaf of basic differential forms and let $Z^n_{\rm bas}$ denote the sheaf of closed basic differential forms. Both are sheaves of $C^{\infty}(M)_{\rm bas}$-modules and fit into the short exact sequence
\begin{equation}\label{eq:CechdeRham}
0 \rightarrow Z^n_{\rm bas} \rightarrow \Omega^{n}_{\bas} \rightarrow Z^{n+1}_{\rm bas} \rightarrow 0.
\end{equation}
Consider the basic \v{C}ech cocycle groups $C_{\bas}^k(\mathcal{U},\mathcal{F})$, for which the maps $U_{i_1\ldots i_k} \rightarrow \mathcal{F}$ are invariant with respect to the $G$-action on $U_{i_1\ldots i_k}$.

By existence of a $G$-invariant partition of unity on $M$, the sheaves $\Omega^n_{\rm bas}$ are fine, and hence $H^i_{\bas}(M,\Omega^n_{\rm bas}) = \set{0}$ for all $i \geq 1$. Therefore the connecting morphism of the above exact sequence induces an isomorphism $H^i_{\bas}(M,Z^{n+1}_{\rm bas}) \rightarrow H^i_{\bas}(M,Z^n_{\rm bas})$. Applying this consecutively, one obtains the isomorphism between $H^{i+1}_{\bas}(M,\rr) \rightarrow H^1_{\bas}(M,Z^i_{\bas})$. As the latter corresponds to $H^{i+1}_{dR,\bas}(M)$ one obtains the desired isomorphism.

We also consider a basic version of the exponential sequence:
\begin{equation}
0 \rightarrow \mathbb{Z}^n \rightarrow C^{\infty}_{\rm bas}(M;\rr^n) \rightarrow C^{\infty}_{\bas}(M;T^n) \rightarrow 0.
\end{equation}
The long exact sequence will induce an isomorphism $\check{H}^1_{\bas}(M,C^{\infty}_{\bas}(M;T^n)) \rightarrow \check{H}^2_{\rm bas}(M;\mathbb{Z}^n)$. Combining this with the \v{C}ech-de Rham isomorphism we obtain a map to $H^2_{\rm bas}(M;\mathbb{Z}^n)$.

\subsection{Connection approach to Haefliger-Salem class}
We now have all the theory in place to compare the curvature of a connection for a standard $T^k$-action with the Haefliger-Salem class:
\begin{proposition}\label{prop:comparison}
Let $\pi : M \rightarrow B$ be a manifold endowed with a standard $T^k$-action and let $\pi' : M' \rightarrow B$ be a manifold endowed with a locally equivalent action. Let $\Theta,\Theta'$ be connection one-forms on $M$ respectively $M'$. Then $[d\Theta-d\Theta'] \in H^2(A_{\partial B};\mathfrak{t}^k)$ is representable by a smooth form, and descends to the Haefliger-Salem class of $M'$ in $\check{H}^2(B;\mathfrak{t}^k)$.
\end{proposition}
\begin{proof}
Let $I_{\abs{D}},I_{\abs{D'}}$ denote the induced elliptic divisors on $M$ respectively $M'$. By Lemma \ref{lem:smoothcurv}, we have that both $[d\Theta]$ and $[d\Theta']$ are representable by smooth forms, therefore we can modify $\Theta$ and $\Theta'$ such that $d\Theta$ and $d\Theta'$ are smooth forms. We will still denote the modified connections by $\Theta$ and $\Theta'$.

Let $\mathcal{U} = \set{U_i}$  be an open cover of $B$, together with equivariant diffeomorphism $h_i : \pi^{-1}(U_i) \rightarrow \pi'^{-1}(U_i)$ provided by the fact that the actions on $M$ and $M'$ are locally equivalent. Furthermore, assume that over $U_i\backslash \partial B$ both $\pi$ and $\pi'$ are trivalisable as principal $T^k$-bundles.

Because the $h_i$ are equivariant, they are morphisms of divisors, and hence we may consider $h_i^*\Theta' \in \Omega^1(\elli|_{\pi^{-1}(U_i)};\mathfrak{t}^k)$. One readily shows that $h_i^*\Theta'$ is a connection on $\pi^{-1}(U_i)$ and thus $\eta_i := \Theta-h_i^*\Theta' \in \Omega^1_{\bas}(\elli|_{\pi^{-1}(U_i)};\mathfrak{t}^k)$ is a basic form. 

Let $f_{ji} \in C^{\infty}(U_i\cap U_j;T^k)$ be the transition functions given by $h_j\comp h_i(x) =(x,f_{ji}(x))$. By Lemma \ref{prop:localconform}, $\eta_i - \eta_j = h_i^*\Theta' - h_j^*\Theta'$ coincides with $\pi^*d\log f_{ji}$ on $\pi^{-1}((U_i \cap U_j)\backslash \partial B)$. By density, we have that $\eta_i - \eta_j = \pi^*d\log f_{ji}$ on the entirety of $\pi^{-1}(U_i\cap U_j)$. As the $\eta_i$ are closed, we conclude that $\set{\eta_i-\eta_j}_{ij}$ defines a cocycle in $\check{C}^1_{\rm bas}(M,Z^1_{\bas})$.

The Haefliger-Salem class is defined by applying the exponential sequence to the class $\set{f_{ji}}$ in $\check{H}^1(\mathcal{U};\mathcal{T}^k)$. Inspecting the \v{C}ech-de Rham isomorphism, we conclude that this class corresponds to the class defined by $\set{\eta_i-\eta_j}_{ij}$. We are left to show that $\set{\eta_i -\eta_j}_{ij} \in \check{H}^1(\mathcal{U},Z^1_{\bas})$ corresponds to $[d\Theta - d\Theta']$. By exactness of \eqref{eq:CechdeRham}, we know that there exist a refinement $\tilde{\mathcal{U}} = \set{\tilde{U}_i}$ of $\mathcal{U}$ and $\tilde{\eta_i} \in \Omega^1_{\bas}(\pi^{-1}(\tilde{U}_i));\mathfrak{t})$ such that $\tilde{\eta}_i - \tilde{\eta}_j = \eta_i - \eta_j$. The class in $H^2_{dR,\bas}(M,\mathfrak{t}^k)$ corresponding to $\set{\eta_i-\eta_j}_{ij}$ is given by $d\tilde{\eta}_i$. 

Note that $\tilde{\eta}_i -\eta_i = \tilde{\eta}_j - \eta_j$, and hence these forms glue to give a globally defined form $\eta \in \Omega^1_{\bas}(M;\mathfrak{t})$. Because $d\eta = d\tilde{\eta}_i -d\eta_i$, and $d\eta_i = d\Theta - h_i^*d\Theta'$, we have that $d\eta$ is smooth. Hence, by Lemma \ref{lem:smoothandexact} we have that $d\eta$ is exact in $H^2(B;\mathfrak{t}^k)$. Consequently $[d\tilde{\eta}_i] = [d\eta_i] \in H^2(B;\mathfrak{t}^k)$. But as $d\eta_i = d\Theta - h_i^*d\Theta'$, we conclude that the class in $H^2_{dR,\bas}(M,\mathfrak{t}^k)$ corresponding to $\set{\eta_i-\eta_j}_{ij}$ is given by $d\Theta - d\Theta'$, which finishes the proof. 
\end{proof}
The description of the Haefliger-Salem class in terms of connections has the following consequence:
\begin{lemma}
Let $\pi : M \rightarrow B$ be a standard $T^k$-action. Then there exists a locally equivalent action $\pi' : M' \rightarrow B$ which admits a flat connection.
\end{lemma}
\begin{proof}
Let $\Theta$ be a connection on $M$, with integral curvature. By Proposition \ref{prop:classification} the class $[d\Theta] \in H^2(B;\zz^k)$ corresponds to a locally equivalent torus action $M' \rightarrow B$. Given a connection $\Theta'$ on $M'$, we have by Proposition \ref{prop:comparison} that $[d\Theta - d\Theta']$ defines the Haefliger-Salem class of $M'$. But as this class is given by $[d\Theta]$, we conclude that $[d\Theta'] = 0$ which finishes the proof.
\end{proof}
The importance of this lemma is that in the classification of standard $T^k$-actions, it is now possible to pick as a reference an action $M_0 \rightarrow B$ with a flat connection. Using this reference, and a locally equivalent action $M' \rightarrow B$, we find that the Haefliger-Salem class of $M_0$ with respect to $M'$ is given exactly by $[d\Theta']$ for some connection $\Theta'$ on $M'$.
\begin{definition}
Let $M \rightarrow B$ be a standard $T^k$-action. The locally equivalent action $M_0 \rightarrow B$ which admits a flat connection is called the \textbf{reference action}.
\end{definition}
When $M \rightarrow B$ is a principal $T^k$-action, the reference action is simply the trivial bundle $B \times T^k$. Contrary to in the principal case, there might be several non locally equivalent standard $T^k$-actions over the same base with flat connections. We will see examples of this in Section \ref{sec:topchange}.

\section{The singular differential geometry of the correspondence space}\label{sec:singgeometry}
Let $M \rightarrow B,\hat{M} \rightarrow B$ be manifolds endowed with standard torus actions. The fibre-product $M \times_B \hat{M} \subset M\times \hat{M}$ fails to be a smooth submanifold. This space plays an important role in $T$-duality, and therefore we will shortly study notions for smooth objects on this space.

Let us start with the simplest example:
\begin{example}
Let $M,\hat{M}$ both denote $\cc$ endowed with the standard $S^1$-action, and $B = \rr_{\geq 0}$ the quotient. Then 
\begin{equation*}
M\times_B \hat{M} = \set{(z,\tilde{z}) : \abs{z} = \abs{\tilde{z}}},
\end{equation*}
and hence we may consider singular coordinates $(r,\theta,\hat{\theta}) \in \rr_{\geq 0} \times T^2$. These coordinates are well-defined except at $r = 0$. We see that $M\times_B \hat{M}$ is a cone over a torus.
\end{example}
The space $\cc \times_{\rr_{\geq 0}} \cc$ will serve as local model for the spaces we want to consider, so we will describe smoothness in terms of it:
\begin{definition}
A function $f : \cc \times_{\rr_{\geq 0}} \cc \rightarrow \rr$ is \textbf{smooth}, if it is the restriction of a smooth function $\tilde{f} \in C^{\infty}(\cc\times\cc)$.
\end{definition}

\begin{definition}
A map $\Phi : \cc \times_{\rr_{\geq 0}} \cc \rightarrow  \locm$ is said to be \textbf{smooth}, if for every $f \in C^{\infty}(\locm)$, $f\comp \Phi$ is smooth.
\end{definition}
Taking products, we can talk about smooth maps and diffeomorphisms of $(\cc \times_{\rr_{\geq 0}} \times \cc)^l \times \rr^k$, and charts and atlases of this form.
%\begin{definition}
%We say that a topological space $N$ is a \textbf{toric fibre product}, if there exists an atlas with charts of the form $(\cc \times_{\rr_{\geq 0}} \cc)^l \times \rr^k$.
%\end{definition}
From the local form of standard $T^k$-actions it immediately follows that:
\begin{lemma}
Let $M,\hat{M}$ be $n$dimensional manifolds with standard $T^k$-actions over the same base $B$. Then $M\times_B\hat{M}$ admits an atlas of coordinate charts of the form $(\cc \times_{\rr_{\geq 0}} \times \cc)^l \times \rr^{n+k-3l}$.
\end{lemma}

Let $p \in M\times_B \hat{M}$ be such that $\pi(p) \in \partial B[l]$ (that is, the isotropy of the $T^{2k}$-action at $p$ is $T^{2l}$). Such a point then has singular coordinates.
\begin{equation}\label{eq:coordinates}
(r_1,\ldots,r_l,\theta_1,\ldots,\theta_l,\hat{\theta}_1,\ldots,\hat{\theta}_l,\psi_1,\ldots,\psi_{k-l},\hat{\psi}_1,\ldots,\hat{\psi}_{k-l},x_1,\ldots,x_{n-k-l}).
\end{equation}
Here the $\set{r_i}$ are one set of radial coordinates on $\cc^l$, where $\set{\theta_i},\set{\hat{\theta}_i}$ are two sets of polar coordinates on $\cc^l$. The $\set{\psi_i}$ and $\set{\hat{\psi}_i}$ are honest coordinates on $T^{k-l}$ and the $\set{x_i}$ are coordinates on $\rr^{n-k-l}$.

\begin{remark}[Induced divisor]
Let $M,\hat{M}$ be manifolds endowed with standard $T^k$-actions and same base $B$ and let $I_{\abs{D}}$ and $I_{\abs{\hat{D}}}$ denote the corresponding elliptic divisors. Then the restriction of $I_{\abs{D}}$ and $I_{\abs{\hat{D}}}$ to $M\times_B \hat{M}$ coincides. In the coordinates of \eqref{eq:coordinates} the induced elliptic divisor on $M\times_B \hat{M}$ is given by $\inp{r_1,\ldots,r_l}$.
\end{remark}

Rather than trying to find a notion of tangent bundle for $M \times_B \hat{M}$, we use the ambient ellipitc tangent bundle on $M \times \hat{M}$ to define a Lie algebroid over $M \times_B \hat{M}$.

The first observation is that the way $M \times_B \hat{M}$ sits inside the smooth manifold $M \times \hat{M}$ is compatible with the Lie algebroid $\elli \boxtimes \ellihat$ in the following manner:
\begin{definition}
Let $\A \rightarrow M$ be a Lie algebroid and $N \subset M$ a subset, and let $\Gamma_N(\A)$ denote the sections of $\A$ tangent to $N$. Then $N$ is \textbf{adapted} to $\A$ if $\Sing(N) \subseteq \set{x\in M : \rho_x \text{ is not of full rank}}$ and the pull-back sheaf $\iota^*_N\Gamma_N(\A)$ is locally free over $\iota^*_NC^{\infty}(M)$.
\end{definition}
As $\iota^*_N\Gamma(\A)$ is locally free, we may view it as a topological vector bundle over $N$. By declaring sections to be smooth if they are restrictions of smooth sections of $\A$, we can view $\A$ as a smooth Lie algebroid over $N$. 
\begin{example}
A subspace $N \subset M$ is adapted to $TM$ if and only if it is an embedded submanifold. In this case the pull-back sheaf is simply given by vector fields on $N$.
\end{example}
The rationale behind this definition is as follows. We allow subspaces to have singularities, but they need to occur precisely when the anchor of $\A$ is not an isomorphism. 
\begin{example}
Let $M = \rr^2$, and let $\A$ be the Lie algebroid generated by $x\partial_x$ and $y\partial_y$. Let $N$ denote the union of the diagonal and anti-diagonal. If $X = fx\partial_x + gy\partial_y$ is tangent to the diagonal, it implies it must preserve the ideal generated by $x-y$. Therefore, $X = f(x\partial_x + y\partial_y)$. As these vector fields are also tangent to the anti-diagonal we conclude that $N$ is an $\A$-adepted subspace, and the restricted Lie algebroid is generated by $x\partial_x + y\partial_y$
\end{example}

From the local form of standard $T^k$-actions, we immediately obtain the following:
\begin{lemma}
Let $M,\hat{M}$ be standard $T^k$-actions over the same base $B$. Then $M \times_B \hat{M}$ is an $\elli \boxtimes \ellihat$-adapted subspace.
\end{lemma}
We denote the induced Lie algebroid on $M \times_B \hat{M}$ by $\elli \times_B \ellihat$.

%\begin{proposition}
%Let $\A\rightarrow M$ and $\A'\rightarrow M'$, $\mathcal{B} \rightarrow B$ Lie algebroids together with Lie algebroid submersions $\pi : \A \rightarrow \mathcal{B}$, $\pi' \rightarrow \mathcal{B}$. Then  $M \times_B M' \subset M \times M'$ is adapted to $\A \boxtimes \A'$. We denote the pull-back $\iota^*_{M\times_B M'}(\A\boxtimes \A')$ by $\A \times_{\mathcal{B}} \A'$ over it.
%\end{proposition}
%\begin{proof}
% \alert{Not entirely sure if this is true}
%\end{proof}

\section{Non prinicipal $T$-duality}\label{sec:Tduality}
Using the Lie algebroids from Section \ref{sec:backgroud} we will extend the notions of $T$-duality to standard torus actions. To proof the existence of $T$-duals we will use the classification of torus actions, and the connection point of view thereof, as described in Section \ref{sec:classification}. Moreover, we will stress the differences with $T$-duality for principal torus bundles; in particular the new source of topology change arising from the singularities of the torus action.
\subsection{$T$-duality}
In this subsection we will recall the notions of $T$-duality for principal torus actions. We will be brief and refer to \cite{BEM04} and \cite{CG11} for more details.
\begin{definition}
Let $M$ and $\hat{M}$ be principal $T^k$-bundles with base $B$ and let $H \in \Omega^3_{T^k}(M),\hat{H} \in \Omega^3_{T^k}(\hat{M})$ be $T^k$-invariant closed three-forms. Consider the fibre-product, $M\times_B \hat{M}$, and the following diagram:
\begin{center}
\begin{tikzcd}
& (M\times_B \hat{M},p^*H-\hat{p}^*\hat{H}) \ar[ld,"p"],\arrow{rd}[swap]{\hat{p}} &\\
(M,H) \ar[rd,"\pi"] & & (\hat{M},\hat{H}) \arrow{ld}[swap]{\hat{\pi}}\\
& B & 
\end{tikzcd}
\end{center}
We say that $(M,H)$ and $(\hat{M},\hat{H})$ are \textbf{$T$-dual} if there exists a $T^{2k}$-invariant two-form $F \in \Omega^2_{T^{2k}}(M\times_B\hat{M})$ such that
\begin{equation*}
dF = p^*H - \hat{p}^*\hat{H},
\end{equation*}
and the restriction
\begin{equation*}
F : \mathfrak{t}_M \times \mathfrak{t}_{\hat{M}} \rightarrow \rr,
\end{equation*}
where $\mathfrak{t}_M,\mathfrak{t}_{\hat{M}}$ denotes the tangent space to the fibres of $p$, $\hat{p}$ respectively, is non-degenerate. We call $M \times_B \hat{M}$ the \textbf{correspondence space}. 
\end{definition}
$T$-duals always exist, provided the three-form is integral:
\begin{theorem}[\cite{BEM04}]
Let $M \rightarrow B$ be a principal $T^k$-bundle, and $H \in \Omega^3_{T^k}(M)$ with integral cohomology class. Then there exists a principal $T^k$-bundle, $\hat{M}$ $T$-dual to $M$.
\end{theorem}

$T$-dual spaces have isomorphic complexes of invariant differential forms:
\begin{theorem}[\cite{BEM04}]
Let $(M,H)$ and $(\hat{M},\hat{H})$ be $T$-dual with $dF = p^*H -\tilde{p}^*\hat{H}$. Then 
\begin{equation*}
\tau :(\Omega^{\bullet}_{T^k}(M),d_H) \rightarrow (\Omega^{\bullet}_{T^k}(\hat{M}),d_{\hat{H}}), \quad \rho \mapsto \int_{T^k}e^F \wedge p^*\rho,
\end{equation*}
is an isomorphism of cochain complexes. Here the integral is over the fibres of $\tilde{p}$.
\end{theorem}

\subsection{Courant algebroids}
$T$-duality induces an isomorphism of invariant sections of the standard Courant algebroid:
\begin{theorem}[\cite{CG11}]\label{th:courantiso}
If $(M,H)$ and $(\hat{M},\hat{H})$ are $T$-dual, then there is an isomorphism of Courant algebroids
\begin{equation*}
\varphi : (TM \oplus T^*M)/T^k \rightarrow (T\hat{M} \oplus T^*\hat{M})/T^k,
\end{equation*}
such that $\tau(v\cdot \rho) = \varphi(v) \cdot \tau(\rho)$, for all $v\in (TM \oplus T^*M)/T^k$ and $\rho \in \Omega^{\bullet}(M)$.
\end{theorem}

The strength of this description is that it immediately allows one to transport structures:
\begin{theorem}[\cite{CG11}]
Let $(M,H)$ and $(\hat{M},\hat{H})$ be $T$-dual spaces. Then any Dirac, generalized
complex, generalized K\"ahler or SKT structure on $M$ which is invariant under the torus
action is transformed via $\varphi$ into a structure of the same kind on $\hat{M}$.
\end{theorem}

We are now ready to introduce the notion of $T$-duality for standard torus actions:
\begin{definition}\label{def:t-duality general}
Let $M,\hat{M}$ be two manifolds endowed with standard $T^k$-actions over the same base $B$ and let $I_{\abs{D}},I_{\abs{\hat{D}}}$ be the induced elliptic divisors. Furthermore, suppose we are given $H \in \Omega^3_{\rm cl}(\ellio), \hat{H} \in \Omega^3_{\rm cl}(\ellit)$ closed and $T^k$-invariant. We consider the following diagram:
\begin{center}
\begin{tikzcd}
& (\ellio \times_B\ellit \rightarrow  M\times_B \hat{M},p^*H-\hat{p}^*\hat{H}) \ar[ld,"p"]\arrow{rd}[swap]{\hat{p}} &\\
(\ellio \rightarrow M,H) \arrow{rd}{\pi} & & (\ellit \rightarrow \hat{M},\hat{H}) \arrow{ld}[swap]{\hat{\pi}}\\
& \A_{\partial B} \rightarrow B & 
\end{tikzcd}
\end{center}
We say that $(M,H)$ and $(\hat{M},\hat{H})$ are \textbf{$T$-dual} if there exists a $T^{2k}$-invariant two-form $F \in \Omega^2(\ellio \times_B \ellit)$ such that $p^*H-\hat{p}^*\hat{H} = dF$ and
\begin{equation}
F : \mathfrak{t}_{M} \otimes \mathfrak{t}_{\hat{M}} \rightarrow \rr,
\end{equation}
is non-degenerate.
\end{definition}
Let us expand on this definition. First $\ellio\times_B \ellit$ is the Lie algebroid induced by viewing $M \times_B \hat{M}$ as an $(\ellio\boxtimes \ellit)$-adapted subspace of $M \times \hat{M}$. The two-form $F$ is a form for this Lie algebroid, and can in practice be thought of as the restriction of a two-form $\tilde{F} \in \Omega^2(\ellio\boxtimes \ellit)$. Here $\mathfrak{t}_{M_i}$ denotes the kernel of the Lie algebroid submersion $\ellio\times_B \ellit \rightarrow \mathcal{A}_{\abs{D_i}}$ induced by $p_i$.

\begin{example}
Consdier $\cc$ with the standard $S^1$-action, and $H = 0$. Then $F = d\theta_1 \wedge d\theta_2$ provides a $T$-dual from $(\cc,0)$ to itself.
\end{example}
A connection one-form $\Theta \in \Omega^1(\elli;\mathfrak{t})$ induces an isomorphism $\elli/S^1 \simeq \mathfrak{t} \oplus \A_{\partial B}$, therefore we obtain a decomposition
\begin{equation*}
\Omega^3_{T^k}(\elli) = \sum_{i=0}^3\Omega^i(\A_{\partial B},\wedge^{3-i}\mathfrak{t}^*).
\end{equation*}
The $T$-duality conditions forces the component of $H$ in $\Omega^0(\A_{\partial B};\wedge^3 t^*)$ to vanish, and the component in $\Omega^1(\A_{\partial B};\wedge^2 t^*)$ to be exact. When considering geometric structures (up to gauge equivalence), it is the cohomology class $[H]$, rather than the form itself, which is of importance. Hence we can impose that we consider $H$ which only have two terms in this decomposition:
\begin{equation}\label{eq:integrality}
H = \inp{\hat{c},\Theta} + h,
\end{equation}
for some $\hat{c} \in \Omega^2(\A_{\partial B};\mathfrak{t}^*)$ and $h \in \Omega^3(\A_{\partial B})$.
Using this decomposition, and the relation of the curvature with the Haefliger-Salem class, we can now address the question of existence of $T$-duals.
\begin{theorem}\label{prop:Texistence}
Let $M \rightarrow B$ be a manifold endowed with a standard $T^k$ and an invariant closed three-form $H \in \Omega^3(\elli)$. If $[\hat{c}]$ as in Equation \eqref{eq:integrality} defines a class in $H^2(B;\mathbb{Z}^k)$, then there exists a $T$-dual $\hat{M} \rightarrow B$.
\end{theorem}
\begin{proof} 
Let $M_0 \rightarrow B$, be the reference torus action of $M$. Using Proposition \ref{prop:classification}, with respect to $M_0$ and $\hat{M}$,  the class $[\hat{c}] \in H^2(B;\zz^k)$ corresponds to a manifold $\hat{M} \rightarrow B$ endowed with a locally equivalent action. Moreover, there exists a connection $\hat{\Theta} \in \Omega^1(\ellihat;\mathfrak{t}^k)$ for which $d\hat{\Theta} = \hat{c}$. Now define
\begin{equation*}
\hat{H} = \inp{c,\hat{\Theta}}+h \in \Omega^3(\ellit),
\end{equation*}
with $c = d\Theta$, for some connection $\Theta \in \Omega^1(M;\mathfrak{t}^k)$. Now $(M,H)$ and $(\hat{M},\hat{H})$ are $T$-dual because
\begin{equation*}
p^*H-\tilde{p}^*\hat{H} = \inp{\hat{c},\Theta} - \inp{c,\hat{\Theta}} = \inp{d\hat{\Theta},\Theta} - \inp{d\Theta,\hat{\Theta}} = - d\inp{\Theta,\hat{\Theta}},
\end{equation*}
and as $\Theta$ and $\hat{\Theta}$ are connection-one-forms, we obtain that $F =- \inp{\Theta,\hat{\Theta}}$ satisfies the requirements of $T$-duality.
\end{proof}

\begin{remark}
Note that it might be preferable to put a direct condition on $H$, rather than the $[\hat{c}]$ in the decomposition. However, this will also impose further conditions on the $h$ term which are not strictly necessary. Still, in practice we will assume that $H$ is smooth and defines an integral cohomology class, which implies the assumptions in Theorem \ref{prop:Texistence} 
\end{remark}

\begin{remark}[Non-uniqueness]
This proposition provides the existence of $T$-duals, but these are by no means unique. Already in the case without fixed points, the construction does not take any torsion of the torus bundle into account. But also changing the specific $F$ used may change the topology of the resulting $T$-dual, see also Bunke-Schick \cite{MR2130624}.
Now the non uniqueness goes even further: because we made use of the classification of Haefliger-Salem, any $T$-dual produced using Proposition \ref{prop:Texistence} will necessarily be locally equivalent to $M$. However, that needn't be the case as we will see momentarily.
\end{remark}
\subsubsection{Topology change from singularities}\label{sec:topchange}
The singularities of the torus actions provide a new source of topology change for $T$-duals:
\begin{proposition}
Let $M,\hat{M}$ be manifolds with standard $T^k$-actions over the same contractible base $B$. Then $(M,H=0)$ and $(\hat{M},\hat{H}=0)$ are $T$-dual.
\end{proposition}
\begin{proof}
Because the base of $M$ and $\hat{M}$ is contractible there exists flat connections $\Theta,\hat{\Theta}$ on $M$ respectively $\hat{M}$. Then, as in the proof of Proposition \ref{prop:Texistence} (the restriction of) $F = -\inp{\Theta,\hat{\Theta}}$ provides a $T$-dual between $(M,0)$ and $(\hat{M},0)$.
\end{proof}
In \cite{OR70} a classification of four-manifolds with effective $T^2$-actions and connected isotropy groups appears. All such manifolds are of the following form:
\begin{itemize}
\item $\#n (S^2\times S^2)$, $n \in \mathbb{N}$
\item $\#n (\cc P^2) \# m(\bar{\cc P^2})$, $n,m \in \mathbb{N}$
\item $S^4$
\end{itemize}
All of these manifolds have a contractible base, so we conclude by the above Theorem that if two have the same base they are $T$-dual. The only invariant of the base is the number of corners, which can be read of from the Euler characteristic of the total space. For instance $2\mathbb{C}P^2,\mathbb{C}P^2\# \bar{\cc P^2}$ and $S^2\times S^2$ are all $T$-dual. The common base of these manifolds is contractible, and hence the actions are necessarily not locally equivalent. For if the actions were locally equivalent, Theorem \ref{th:HS1} would imply that these spaces are all diffeomorphic.

Just as in principal $T$-duality, we obtain an isomorphism of differential complexes. The proof is omitted as it is identical with only the minor change that instead of developing the operation of push-forward/fiber integration we use the equivalent operation of ``interior product with a top degree multi-vector field'' which is readily understood in the Lie algebroid context.
\begin{proposition}\label{prop:complexiso}
Let $M,\hat{M}$ be manifolds endowed with standard $T^k$-actions over the same base $B$. Let $H \in \Omega^3_{\rm cl}(\ellio),\hat{H} \in \Omega^3_{\rm cl}(\ellit)$, and let $F \in \Omega^2(\ellio \times_B \ellit)$ provide a $T$-dual. Then 
\begin{equation}
\tau : \Omega^{\bullet}_{T^k}(\ellio) \rightarrow \Omega^{\bullet}_{T^k}(\ellit), \quad \tau(\rho) =  \iota_{X_1\wedge \cdots \wedge X_n} (e^F\wedge \rho),
\end{equation}
is an isomorphism of differential complexes $(\Omega^{\bullet}_{T^k}(M;\ellio),d_{H})$ and $(\Omega^{\bullet}_{T^k}(\hat{M};\ellit),d_{\hat{H}})$. Here $X_1,\ldots,X_n \in \Gamma(\ellio)$ are the infinitesimal generators of the torus action on $M$.
\end{proposition}

%\subsection{H}
%We realize that, because elliptic forms are used in the construction of $\hat{H}$ it needn't be the case that if $H$ is smooth so if $\hat{H}$. Indeed, simply pick any connection one-form $\theta$, for which $d\theta$ is not smooth. For instance, on $\cc^2$ with the standard action, one may consider the connection one-form $(d\theta_1,d\theta_2 - r_2^2d\log r_1)$. Independent of the other choices, this will imply that $H$ has a $d(r_2^2)\wedge d\log r_1 \wedge d\hat{\Theta}_2$-term, and hence is not smooth.
%
%However, for the geometric structures on $\ellihat$, it is not $\hat{H}$ which is relevant, but rather its cohomology class in elliptic cohomology. This razes the question, whether it is always possible to represent $[\hat{H}]$ by a smooth form.

Moreover, we may consider the standard $\elli$-Courant algebroid, $\elli \oplus \elli^*$, and obtain Theorem \ref{th:courantiso} in our setting. Again, the proof is identical to the principal case and thus omitted.
\begin{theorem}\label{th:courantisoelliptic}
Let $M,\hat{M}$ be manifolds endowed with standard $T^k$-actions over the same base $B$. Let $H \in \Omega^3_{\rm cl}(\ellio),\hat{H} \in \Omega^3_{\rm cl}(\ellit)$, and assume that $(M,H)$ and $(\hat{M},\hat{H})$ are $T$-dual. Then there is an isomorphism between Courant algebroids $(\elli \oplus \ellit^*)/T^k  \rightarrow M$ and $(\elli \oplus \ellit^*)/T^k  \rightarrow \hat{M}$.
\end{theorem}
%\begin{proof}
%Given $X + \xi \in TM \oplus T^*M / T^k$, let $\hat{X} \in \Gamma(\ellio \times_B \ellit$ be the unique vector field such that $\xi(Y) = F(\hat{X},Y)$ for all $Y \in \mathfrak{t}^k_M$. Then
%\begin{equation*}
%\phi(X+\xi) = \tilde{p}_*(\hat{X}) + p^*\xi - F(\hat{X}),
%\end{equation*}
%will be the desired isomorphism. \alert{bit more, or not? It is identical to the old situation.}
%\end{proof}

Using this isomorphism we are able to transport geometric structures on $\elli \rightarrow M$ to geometric structures on $\ellit \rightarrow \hat{M}$:
\begin{theorem}\label{th:transportgeom}
Let $(M,H)$ and $(\hat{M},\hat{H})$ be $T$-dual. Then any Dirac, generalized complex or SKT structure on $\elli \rightarrow M$ invariant under the torus action will be transported to a structure of the same kind on $\elli \rightarrow \hat{M}$.
\end{theorem}

However, there is one key difference compared to the case of principal $T$-duality. Given a $T^k$-invariant generalized complex structure on $M$ which lifts to a generalized complex structure on $\elli \rightarrow M$, we can transport it to a generalized complex structure on $\ellihat \rightarrow \hat{M}$. However, a prior there is no guarantee that it again descends to a generalized complex structure on $M$. We illustrate this phenomenon by the following examples:
\begin{example}[$\mathbb{\cc}$]
Consider $M = \cc$ with the standard $S^1$-action, and endow it with the elliptic symplectic form $d\log r \wedge d\theta = \frac{dx \wedge dy}{x^2+y^2}$. If we let $\hat{M} = \cc$ denote another another copy of $\cc$, we have that $F= d\theta \wedge d\hat{\theta}$ provides a $T$-dual. Under this $T$-duality, the elliptic symplectic structures $e^{id\log r \wedge d\theta}$ is send to spinor $d\log \hat{z}$. This spinor corresponds to the complex structure on $\elli = \set{r\partial_r,\partial_{\theta}}$ defined by $r\partial_r \mapsto \partial_{\theta}$. It is immediate to see that this complex structure is a lift of the standard complex structure on $\cc$. Hence we conclude that the elliptic symplectic form, which is not a generalized complex structure, is $T$-dual to the standard complex structure.
\end{example}
The next example will illustrate that $T$-duality for standard torus actions is very sensitive to the specific $F \in \Omega^2(\elli \times_B \ellihat)$ chosen:
\begin{example}[$\mathbb{\cc}^2$]\label{ex:c2}
Consider $M = \cc^2$ with coordinates $(z_1,z_2)$ and $\hat{M} = \cc^2$ with coordinates $(\hat{z}_1,\hat{z}_2)$. Endow both with the standard $T^2$-action. There are ample choices for a $T$-duality between $(M,0)$ and $(\hat{M},0)$. Indeed, any of the following two-forms will suffice:
\begin{itemize}
\item $F_{1,\pm} = d\theta_1 \wedge d\hat{\theta}_1 \pm d\theta_2 \wedge d\hat{\theta}_2$.
\item $F_{2,\pm} = -d\theta_1 \wedge d\hat{\theta}_2 \pm d\theta_2 \wedge d\hat{\theta}_1$.
\end{itemize}
Although each of these forms satisfy the conditions necessary of $T$-duality, the behaviour with respect to specific generalized complex structures is quite different.

We consider the stable generalized complex structure given by $\rho = z_1z_2 + dz_1 \wedge dz_2$. The corresponding elliptic symplectic form is $\omega = d\log r_1 \wedge d\theta_2 + d\theta_1 \wedge d\log r_2$ with spinor $\rho = 1 + id\log r_1 \wedge d\theta_2 + id\theta_1 \wedge d\log r_2 - d\log r_1 \wedge d\theta_1 \wedge d\log r_2 \wedge d\theta_2$. 

By Theorem \ref{th:transportgeom} we know that $T$-duality with any of these forms will provide generalized complex structures on the elliptic tangent bundle. However, a direct computation will show that only $F_{2,+}$ provides a generalized complex structure on $M$. In that case we have that the $T$-dual generalized complex structure is given by:
%
%
%
%We will perform the $T$-duality circle by circle, so first for $\theta_1$ with $F_{1,\pm}$:
%\begin{equation}
%d\hat{\Theta}_1 + id\hat{\Theta}_1 \wedge d\log r_1 \wedge d\hat{\Theta}_2 + id\log r_2 + d\log r_1 \wedge d\log r_2 \wedge d\theta_2,
%\end{equation}
%and now with $\theta_2$:
%\begin{equation}
%\rho_{1,\pm} = \pm d\ttheta_2 \wedge d\ttheta_1 + i d\ttheta_1 \wedge d\log r_1 \pm id\ttheta_2 \wedge d\log r_2 +  d\log r_1 \wedge d\log r_2.
%\end{equation}
%And now we will perform similarly with $F_{2,\pm}$,
%\begin{equation}
%d\ttheta_2 + id\ttheta_2 \wedge d\log r_1 \wedge d\theta_2 + id\log r_2 +  d\log r_1 \wedge d\log r_2 \wedge d\theta_2,
%\end{equation}
%and on $\theta_2$:
%\begin{equation}
%\rho_{2,\pm} = \pm d\ttheta_1 \wedge d\ttheta_2 +   id\ttheta_2 \wedge d\log r_1 \pm id\ttheta_1 \wedge d\log r_2 + d\log r_1 \wedge d\log r_2.
%\end{equation}
\begin{equation*}
\rho_{2,+} = d\log z_1 \wedge d\log z_2.
\end{equation*}
One immediately checks that $\rho_{2,+}$ is a lift of the standard complex structure on $\cc$. In conclusion, this elliptic symplectic structure on $\cc^2$ is $T$-dual to the standard complex structure on $\mathbb{C}^2$.

The story for the real part of $\rho$ is quite different. Consider the elliptic symplectic form $\omega = d\log r_1 \wedge d\log r_2 - d\theta_1 \wedge d\theta_2$, with spinor 
\begin{equation*}
\rho = 1 + i d\log r_1 \wedge d\log r_2 - id\theta_1 \wedge d\theta_2 + d\theta_1 \wedge d\theta_2 \wedge d\log r_1 \wedge d\log r_2.
\end{equation*}
Again, one can perform a direct computation showing that only $F_{1,-}$ will provide a generalized complex structure on $M$. In that case the $T$-dual is given by:
\begin{equation*}
-ie^{-i(d\log r_1 \wedge d\log r_2 + d\hat{\theta}_1 \wedge d\hat{\theta_2})}
\end{equation*}
%
%We will perform $T$-duality circle by circle, so first for $\theta_1$ with $F_{1,\pm}$:
%\begin{equation}
%d\hat{\Theta}_1 + id\hat{\Theta}_1 \wedge d\log r_1 \wedge d\log r_2 = i d\theta_2 + d\theta_2 \wedge d\log r_1 \wedge d\log r_2,
%\end{equation}
%and now with $\theta_2$:
%\begin{align}
%&\pm d\hat{\Theta}_2 \wedge d\hat{\Theta}_1 \pm i d\tilde{\theta_2} \wedge d\hat{\Theta}_1 \wedge d\log r_1 \wedge d\log r_2 -i d\log r_1 \wedge d\log r_2\\
%&= -i(1-id\log r_1 \wedge d\log r_2 \pm i d\hat{\Theta}_1 \wedge d\tilde{\theta_2} \pm d\hat{\Theta}_1 \wedge d\hat{\Theta}_2 \wedge d\log r_1 \wedge d\log r_2\\
%&= -ie^{-i(d\log r_1 \wedge d\log r_2 \pm d\hat{\Theta}_1 \wedge d\tilde{\theta_2})},
%\end{align}
In conclusion we see that $\rho$ is $T$-dual to itself. \hfill \qedhere

%Now we can perform some B-field transformations with $B \in \Omega^2(\elli)$ (note that this is allowed, given the fact that a priori we really see this as $\elli$-generalized complex structures. Alternatively, one can put the B-field in the choice of $F$).
%We find
%\begin{align*}
%\rho_{1,\pm} &= e^{d\log r_1 \wedge d\log r_2 \pm d\ttheta_2 \wedge d\ttheta_1} \wedge (1 + id\ttheta_1 \wedge d\log r_1 \pm i d\ttheta_2 \wedge d\log r_2 \alert{\mp} d\log r_1 \wedge d\log r_2 \wedge d\theta_1 \wedge d\theta_2)\\
%&= e^{d\log r_1 \wedge d\log r_2 \pm d\ttheta_2 \wedge d\ttheta_1}(e^{i(d\ttheta_1 \wedge d\log r_1 \pm  d\ttheta_2 \wedge d\log r_2)})
%\end{align*}
%
%\begin{align*}
%\rho_{2,\pm} &= e^{d\log r_1 \wedge d\log r_2 \pm d\theta_1 \wedge d\theta_2}(1+i d\ttheta_2 \wedge d\log r_1 \pm d\ttheta_1 \wedge d\log r_2 \alert{\pm} d\log r_1 \wedge d\log r_2 \wedge d\ttheta_1 \wedge d\ttheta_2)\\
%&=e^{d\log r_1 \wedge d\log r_2 \pm d\theta_1 \wedge d\theta_2}(e^{i(d\ttheta_1 \wedge d\log r_2 \pm  d\ttheta_2 \wedge d\log r_1)})
%\end{align*}
%
%
%
%
%These computations have several conclusions: First, although $\rho$ is induced by a generalized complex structure on $M$, we see that $\rho_{1,\pm}$ is only an $\elli$-generalized complex structure as the residue conditions are violated. Nevertheless, it defines an elliptic-symplectic form.
%
%However, $\rho_{2,-}$ is B-field equivalent to $\rho$, and thus itself induced by a generalized complex structure. Note that $\rho_{2,+}$ satisfies the residue conditions up to sign, and thus is a locally complex elliptic symplectic structure.
\end{example}

\section{Blow-ups}\label{sec:blowups}
In this section we will study the interaction between elliptic geometry and blow-ups. We will work with the formalism of generalized complex blow-ups introduced in \cite{BCD16}. We will show that blow-ups of submanifolds of the degeneracy locus of an elliptic divisor induce Lie algebroid submersions between the respective elliptic tangent bundles. Torus actions are preserved by blow-ups at fixed points, and hence we may consider the relation with $T$-duality. We will show that $T$-duality and blowing up commutes.

\subsection{Blow-ups in generalized complex geometry}
The contents of this subsection appear in \cite{BCD16}, we recall them for completeness. 

As we do not assume any holomorphicity of our manifolds, more data is necessary in constructing a blow-up:
\begin{definition}[\cite{BCD16}]\label{def:holideal}
Let $M^n$ be a smooth manifold, and let $Y$ be a closed embedded submanifold of codimension $2l$. A \textbf{holomorphic ideal} for $Y$ is a sheaf of ideals $I_Y \subset C^{\infty}(M;\cc)$ such that:
\begin{itemize}
\item $I_Y|_{M\backslash Y} = C^{\infty}(M;\cc)|_{M\backslash Y}$.
\item For every $y \in Y$, there is a neighbourhood $U$ with coordinates $z_1,\ldots,z_l,x_{2l+1},\ldots,x_n : U \rightarrow \cc^l\times \rr^{n-2l}$ such that $Y \cap U = \set{z_1=\ldots=z_l = 0}$ and $I_Y|_U = \inp{z_1,\ldots, z_l}$. \hfill \qedhere
\end{itemize}
\end{definition}

As shown in \cite{BCD16}, a submanifold $Y$ admits a holomorphic ideal if and only if the normal bundle $NY$ admits a complex structure.

\begin{theorem}[\cite{BCD16}]
Given a submanifold $Y$ and a holomorphic ideal $I_Y$, there exists a, unique up to isomorphism, manifold $\tilde{M}$ together with a map $p : \tilde{M} \rightarrow M$  such that $p^*I_Y$ is a divisor satisfying the following universal property: For any smooth map $f : X \rightarrow M$, such that $f^*I_Y$ is a divisor, there exists a unique $\tilde{f} : X \rightarrow \tilde{M}$ such that the following diagram commutes:
\begin{center}
\begin{tikzcd}
X \ar[r,"\tilde{f}"] \ar[rd,"f"] & \tilde{M} \ar[d,"p"]\\
& M
\end{tikzcd}
\end{center}
\end{theorem}
We call $\tilde{M}$ the \textbf{blow-up of $I_Y$ in $M$}, and $p : \tilde{M} \rightarrow M$ the blow-down map. The natural manifolds one can blow-up in generalized complex geometry are the following:
\begin{definition}[\cite{BCD16}]
Let $(M,\mathbb{J},H)$ be a generalized complex manifold. A submanifold $Y$ is called \textbf{generalized Poisson} if 
\begin{equation*}
\mathbb{J} N^*Y = N^*Y. \hfill \qedhere
\end{equation*}
\end{definition}
An immediate consequence of this definition is that $Y$ is a Poisson submanifold of the underlying Poisson structure $\pi_{\mathbb{J}}$. Moreover, if $Q$ is a holomorphic Poisson submanifold and $\mathbb{J}_{Q}$ is the associated generalized complex structure, a submanifold $Y$ is a generalized Poisson submanifold if and only if it is a holomorphic Poisson submanifold.

By definition, the generalized complex structure induces a complex structure on the normal bundle of a generalized Poisson submanifold. Hence:
\begin{lemma}[\cite{BCD16}]
Let $(M,\mathbb{J},H)$ be a generalized complex manifold and $Y$ a generalized Poisson submanifold. Then there exist a canonical holomorphic ideal $I_Y$ for $Y$ for which the induced complex structure on $NY$ coincides with the one induced by $\mathbb{J}$.
\end{lemma}

Given that a generalized Poisson submanifold gives rise to a holomorphic ideal, and hence a canonical blow-up, the next step is to determine when the generalized complex structure lifts to the blow-up. For this we need the following notion:
\begin{definition}
A Lie algebra $\mathfrak{g}$ is said to be degenerate if the map $\Lambda^3(\mathfrak{g}) \rightarrow \rm{Sym}^2(\mathfrak{g})$
\begin{equation*}
x\wedge y \wedge z \mapsto [x,y] \wedge z,
\end{equation*}
vanishes.
\end{definition}
Given a generalized Poisson submanifold $Y \subset (M,\mathbb{J})$, it is in particular a Poisson submanifold and hence one obtains a fibre-wise Lie algebra structure on $N^*Y$. This Lie algebra structure captures whether the generalized complex structure lifts to the blow-up:
\begin{theorem}[\cite{BCD16}]
Let $(M,\mathbb{J},H)$ be a generalized complex manifold and $Y$ a generalized Poisson submanifold. Then there exists a generalized complex structure $\tilde{\mathbb{J}}$ on the blow-up $p : \tilde{M} \rightarrow M$, with respect to the canonical holomorphic ideal $I_Y$, for which $p$ is a generalized holomorphic map, if and only if the Lie algebra structure on $N^*Y$ is degenerate.
\end{theorem}
We may blow-up components of the degeneracy locus of a stable generalized complex structure:
\begin{lemma}\label{lem:stableblow}
Let $(M,\mathbb{J})$ be a stable generalized complex manifold, and let $Y \subset D(l)$ be an irreducible component. Then $Y$ is a generalized Poisson submanifold for which the fibre-wise Lie algebra structure on $N^*Y$ is trivial.
\end{lemma}
\begin{proof}
First we note that $Y$ is a union of symplectic leaves of the underlying Poisson structure $\pi_{\mathbb{J}}$, and hence a Poisson submanifold. For every $p \in Y$, there exists a neighbourhood $U$, with a complex structure and a holomorphic Poisson structure $Q$ inducing the generalized complex structure (up to gauge equivalence). We have that $Y \cap U$, can be written as the zero-set of some of the complex coordinates, and hence it is a complex submanifold of $U$. Therefore, $Y \cap U$ is a holomorphic Poisson submanifold with respect to the holomorphic Poisson structure $Q$, and consequently a generalized Poisson submanifold of $\mathbb{J}$ on $U$. As being a generalized Poisson submanifold is a point-wise condition, this shows that $Y$ is a generalized Poisson submanifold of $M$. 

Using a tubular neighbourhood of $Y$, we may transport $\pi_{\mathbb{J}}$ to a Poisson structure on $NY$ which we will still denote in the same fashion. The Lie algebra structure on $N^*Y$ is then defined by

\begin{equation*}
[s,t] := \pi_{\mathbb{J}}(d\hat{s},d\hat{t}), \text{for } s,t \in \Gamma(N^*Y),
\end{equation*}
where $\hat{s},\hat{t}$ denotes the fibre-wise linear function associated to the section. Or, $[s,t] = \pi^{\sharp}_{\mathbb{J}}(d\hat{s})(d\hat{t})$. 

Because $\pi_{\mathbb{J}}$ is an elliptic symplectic structure it lifts to an element of $\Gamma(\wedge^2\elli)$. In particular, this shows that the image of $\pi_{\mathbb{J}}^{\sharp}$ lands in $\elli$. But as $Y$ is contained in a stratum of the divisor $D$, this in particular implies that $\im \pi_{\mathbb{J}}^{\sharp} \cap NY = \set{0}$. This shows that the Lie algebra structure on $N^*Y$ is trivial.
\end{proof}
We conclude that irreducible components of the degeneracy locus of a stable generalized complex structure may be blown-up in a generalized complex fashion.

\subsection{Blowing up for the elliptic tangent bundle}
There is yet another explanation why one can blow up irreducible components of the degeneracy locus of a stable generalized complex structure. Namely, the elliptic tangent bundle is very well-behaved under blow-ups. To state a result independent of a generalized complex structure, we need that the holomorphic ideal is adapted to the elliptic ideal:
\begin{definition}
Let $M^n$ be endowed with an elliptic divisor $I_{\abs{D}}$, and $Y^{n-2l}$ a closed embedded submanifold. We say that a holomorphic ideal $I_Y$ for $Y$ is \textbf{adapted} to $I_{\abs{D}}$ if there exists local coordinates $z_1,\ldots,z_k,x_{2k+1},\ldots,x_{n-2k}$ for which $z_1,\ldots,z_l$ are as in Definition \ref{def:holideal} and $I_{\abs{D}}|_U = \inp{\abs{z_1}^2\cdot \ldots \cdot \abs{z_k}^2}$.
\end{definition}
Note that this definition in particular implies that $Y \subset D$. Given such a holomorphic ideal, we may blow-up the elliptic divisor to obtain:

%%%%%%%%%%%%%%%%%%%%%%%%%%%%%%%%
\begin{theorem}\label{th:blowupelliptic}
Let $(M,I_{\abs{D}})$ be a manifold endowed with an elliptic divisor, and let $Y \subset D(k)$ be a irreducible component. Assume that we are given a holomorphic ideal $I_Y$ for $Y$ adapted to $I_{\abs{D}}$. Then there exists an elliptic ideal $I_{\abs{\tilde{D}}}$ on the blow-up $\tilde{M}$ such that the blow-down map induces a fibre-wise isomorphism: 
\begin{align*}
p_* : (\mathcal{A}_{\abs{\tilde{D}}})_x \overset{\simeq}{\longrightarrow} (\elli)_{p(x)},
\end{align*}
for all $x \in \tilde{M}$.
\end{theorem}
\begin{proof}
To study the blow-down map we need to use local coordinates describing it. Given $l\geq k$ let $x \in Y \cap D[l]$, let $U = \cc^l \times \rr^m$ be a local chart with coordinates $(z_1,\ldots,z_l,x_1,\ldots,x_m)$, with the $z_i$ as in Definition \ref{def:holideal}. Let $\tilde{U} = \tilde{\cc}^l \times \rr^m$, and $p = (p',\text{Id}) : \tilde{U} \rightarrow U$ be the blow-down map. The manifold $\tilde{\cc}^l$ has a cover of $l$-coordinate charts given by 
\begin{align*}
(v_1,\ldots,v_{i-1},z_i,v_{i+1},\ldots,v_l) &\leftrightarrow (z_i \cdot (v_1,\ldots,v_{i-1},1,v_{i+1},\ldots,v_l),\\
&[v_1:\ldots:v_{i-1}:1:v_{i+1}:\ldots:v_l]).
\end{align*}
In these coordinate charts the blow-down map is given by
\begin{align*}
p' : (v_1,\ldots,v_{i-1},z_i,v_{i+1},\ldots,v_l) &\mapsto (z_iv_1,\ldots,z_iv_{i-1},z_i,z_iv_{i+1},\ldots,z_iv_l).
\end{align*}
With these coordinates at hand, we can consider the local form of $p^*I_{\abs{D}}$. Note that this does not provide an elliptic divisor yet, however $I_{\abs{D}} := \sqrt{p^*I_{\abs{D}}}$ does. Because a vector field preserves $p^*I_{\abs{D}}$ if and only if it preserves $\sqrt{p^*I_{\abs{D}}}$ we find that the blow-down map induces a Lie algebroid map $p_* : \mathcal{A}_{\abs{\tilde{D}}} \rightarrow \elli$.

To check that the blow-down map $p$ induces a fibre-wise isomorphism, it suffices to show that it pulls-back a volume element $(\elli)_{p(x)}$ to a volume element of $(\mathcal{A}_{\abs{\tilde{D}}})_x$. In the above coordinates, the following would define a local volume form on $(\elli)_x$:
\begin{align*}
\Omega = d\log r_1 \wedge d\theta_1 \wedge \cdots \wedge d\log r_l \wedge d\theta_l \wedge dx_1 \wedge \cdots \wedge dx_m.
\end{align*}
Therefore
\begin{align*}
p^*\Omega = d\log \tilde{r}_1 \wedge d\hat{\theta}_1 \cdots \wedge d\log \tilde{r}_l \wedge d\hat{\theta}_l \wedge d\log r_{l+1} \wedge d\theta_{l+1} \wedge \cdots d\log r_k \wedge d\theta_k \wedge  dx_1 \wedge \cdots dx_m,
\end{align*}
where the $(\tilde{r}_j,\hat{\theta}_j)$ are the polar coordinates corresponding to $(v_1,\ldots,v_{i-1},z_i,v_{i+1},\ldots,v_l)$. We conclude that $p$ indeed induces the fibrewise isomorphism as required.
\end{proof}

\begin{example}[Blowing up points]\label{ex:points}
Let $M$ be an orientable manifold and let $Y = \set{p}$ be a point. Then we can always find complex coordinates around $p$, $z_1,\ldots,z_l$, such that $I_Y = \inp{z_1,\ldots,z_l}$ is a holomorphic ideal. The topology of the blow-up will depend on the parity of $l$, and whether the coordinates $z_1,\ldots,z_l$ are orientation preserving. If $l$ is odd, the blow-up will be $M\# \mathbb{C}P^l$, if $l$ is even it will be $M\# \bar{\mathbb{C}P^l}$ if the coordinates are oriented and $M\#\mathbb{C}P^l$ if they are not.
\end{example}

\begin{example}[Global normal crossing]
Let $I_{\abs{D}}$ be a global normal crossing divisor, and let $Y \subset D(l)$ be an irreducible component. We have that $Y$ must then necessarily be of the form $D_1\cap \cdots \cap D_l$, for some of the components of $D$. If $D[1]$ is co-orientable, it follows that each $D_i$ is co-orientable as well, and one can show that there exists a complex structure on $ND_i$, inducing the elliptic divisor (see \cite{CKW20}). Combining these complex structures to a complex structure on $NY$, gives a holomorphic ideal for $Y$ which is compatible with $I_{\abs{D}}$.
\end{example}

%Theorem \ref{th:blowupelliptic}, also works in the holomorphic category (see also Remark \ref{rm:hologdiv}, where the statement is well-known:
%\begin{corollary}
%Let $M$ be complex manifold, with a holomorphic log divisor $D \subset M$. Then the complex blow-up induces a fibre-wise isomorphism 
%\begin{align*}
%p_* : T(-\log \tilde{D})_x \overset{\simeq}{\longrightarrow} T(-\log D)_{p(x)},
%\end{align*}
%for all $x \in \tilde{M}$.
%\end{corollary}

Given a generalized complex structure on any of the Lie algebroids $\elli$ may lift it to a structure of the same kind on the blown-up Lie algebroid. But similarly as for $T$-duality, it might happen that said structure does not descend to give a generalized complex structure on $\tilde{M}$. However, when the holomorphic ideal is induced by the generalized complex structure, this will always be the case:
\begin{lemma}
Let $(M,\mathbb{J})$ be a stable generalized complex manifold, and let $Y \subset D(k)$ be an irreducible component. Let $\omega \in \Omega^2(\elli)$ denote the elliptic symplectic structure corresponding to $\mathbb{J}$. Let $I_Y$ be the canonical holomorphic ideal for $Y$. Then this ideal is adapted to $I_{\abs{D}}$, and the blow-up $p^*\omega$ will correspond to the generalized complex blow-up of $\mathbb{J}$ as in Lemma \ref{lem:stableblow}.
\end{lemma}
We conclude that the point of view of blowing-up using the elliptic tangent bundle results in the same blow-up as in \cite{BCD16}. However, Theorem \ref{th:blowupelliptic} can be used in a larger context: for instance, elliptic symplectic forms not corresponding to generalized complex structures may be blown-up using this result.

\subsubsection{Complex structures on the blow-up}
Let $(M,I_{\abs{D}})$ is a manifold endowed with an elliptic divisor, and $\mathbb{J}$ is a generalized complex structure on $M$, which lifts to $\elli$. As pointed out before, if we are given $Y \subset D(k)$ an irreducible component, and $I_Y$ a holomorphic ideal, then the blow-up of $\mathbb{J}$ induces a generalized complex structure on $\elliti$, however this needn't descend to a generalized complex structure on $\tilde{M}$. The following example shows this can already happen with complex structures:

\begin{example}[Complex structures on $n\mathbb{C}P^2 \# m\bar{\mathbb{C}P^2}$]\label{ex:complexstrucs}
Consider the standard complex structure $J$ on $\mathbb{C}P^2$, and $\mathcal{O}(3)$ and the section $s$ corresponding to the polynomial $z_0z_1z_2$. Then $I_D := s^*(\Gamma(\mathcal{O}(3))) \subset C^{\infty}(\mathbb{C}P^2,\cc)$ defines a complex log divisor with degeneracy locus $D = \set{[0:z_1:z_2]} \cup \set{[z_0:0:z_2]} \cup \set{[z_0:z_1:0]}$, and let $I_{\abs{D}}$ be the associated elliptic divisor. As $D$ is a complex submanifold of $\mathbb{C}P^2$, $J$ will preserve the log tangent bundle $A_D$, and consequently induce a complex structure $J_{\elli}$ on $\elli \rightarrow \cc P^2$.

We may also consider $\bar{\cc P^2}$, with the same elliptic divisor $I_{\abs{D}}$. Although this is not an almost complex manifold, it does admit a complex structure on $\elli$. Indeed, one may show that in affine coordinates $(w_1,w_2)$, the spinor $\rho = d\log w_1 \wedge d\log \bar{w}_2$ gives rise to a global complex structure on $\elli \rightarrow \bar{\cc P^2}$.

Now, let $Y = \set{{\rm pt}} \subset \cc P^2$ be any of the vertices of $D$, and pick a holomorphic ideal for $Y$. As described in Example \ref{ex:points}, the corresponding total space will be either $\cc P^2\# \cc P^2$ or $\cc P^2 \# \bar{\cc P^2}$. In either case, we obtain a complex structure $J_{\elliti}$ on the blown-up Lie algebroid $\A_{\abs{\tilde{D}}}$. 

On $\cc P^2 \# \bar{\cc P^2}$, the blow-up coincides with the usual complex blow-up and $J_{\elliti}$ is simply a lift of the complex structure on $\cc P^2 \# \bar{\cc P^2}$. However, $\cc P^2\# \cc P^2$ is not even an almost complex manifold, and hence the induced complex structure on $J_{\elliti}$ does not descend to a (generalized) complex structure on $\cc P^2\# \cc P^2$.
Continuing inductively, we obtain complex structures on elliptic tangent bundles over $n \cc P^2 m \#\cc P^2$, for all $n,m \in \mathbb{N}$.
\end{example}

\subsection{$T$-duality and blow-up}
We now study how $T$-duality interacts which blow-ups. First we recall that toric actions are preserved under blow-ups:
\begin{lemma}
If $M^{2n}$ admits a standard $T^n$ action with quotient map $\pi\colon M \to B$, and $x \in M$ is a fixed point, let  $\tilde{M}$ either blow up of $M$ at $p$, that is $\tilde{M} = M\#\cc P^n$ or $M\#\bar{\cc P^n}$, with blow-down map $q\colon \tilde{M} \to M$. Then $\tilde{M}$ admits a  standard $T^n$-action for which the $q$ is equivariant, and the orbit space on $\tilde{M}$ is $\tilde{B}$, the real blow-up of $B$ at $\pi(x)$. These maps are related by the following diagram
\[\xymatrix{\tilde{M} \ar[r]^{q}\ar[d]^{\tilde\pi} & M\ar[d]^{\pi} \\
\tilde{B} \ar[r]& B}\]
\end{lemma}

The fact that the blow-down map induces a fibre-wise isomorphism between elliptic tangent bundles immediately implies the following:
\begin{proposition}[$T$-duality commutes with blow-up]\label{prop:blow-upT}
Let $M^{2n}$ and $\hat{M}^{2n}$ be endowed with standard $T^n$-actions over the same base $B$, and let $I_{\abs{D}},I_{\abs{\hat{D}}}$ denote the induced elliptic divisors. Let $H \in \Omega^3_{\rm cl}(\ellio), \hat{H} \in \Omega^3_{\rm cl}(\ellit)$, and $F\in \Omega^2(\ellio \times_B \ellit)$ provide a $T$-dual between $(M,H)$ and $(\hat{M},\hat{H})$.

Let $x\in M, \hat{x} \in \hat{M}$ be fixed points for the action with $\pi(x) = \hat{\pi}(\hat{x})$, and let $\tilde{M}$ and $\tilde{\hat{M}}$ denote either blow-up of $M$ and $\hat{M}$. Let $q : \tilde{M} \rightarrow M$, $\hat{q} : \tilde{\hat{M}} \rightarrow \hat{M}$ denote the blow-down maps. Then we have a canonical $T$-duality between the blow-ups:
\begin{center}
\begin{tikzcd}
& (\tilde{M}\times_{\tilde{B}} \tilde{\hat{M} }),p^*q^*H-\hat{p}^*\hat{q}^*\hat{H}) \ar[ld,"p"],\ar[rd,"\hat{p}"] &\\
(\tilde{M},q^*H) \ar[rd,"\tilde{\pi}"] & & (\tilde{\hat{M}},\hat{q}^*\hat{H}) \ar[ld,"\tilde{\hat{\pi}}"]\\
& \tilde{B} & 
\end{tikzcd}
\end{center}
Consequently, the following diagram commutes:
\begin{center}
\begin{tikzcd}
\Omega^{\bullet}(\mathcal{A}_{\abs{\tilde{D}}}) \ar[r,"\tilde{\tau}"] & \Omega^{\bullet}(\mathcal{A}_{\abs{\tilde{\hat{D}}}}) \\
\Omega^{\bullet}(\mathcal{A}_{\abs{D}}) \ar[r,"\tau"] \ar[u,"q^*"] & \Omega^{\bullet}(\mathcal{A}_{\abs{\hat{D}}}) \ar[u,"\hat{q}^*"]
\end{tikzcd}
\end{center}

\end{proposition}
\begin{proof}
By Theorem \ref{th:blowupelliptic} $q$ and $\hat{q}$ induce fibre-wise isomorphisms between elliptic tangent bundles. Let $Q : \tilde{M}\times_{\tilde{B}} \tilde{\hat{M}} \rightarrow M \times_B \hat{M}$ denote the induced map between the correspondence spaces. Consequently $Q$ induces a fibre-wise isomorphism between the elliptic tangent bundles on the correspondence spaces. Therefore $Q^*F$ will satisfy the equation for $T$-duality.
\end{proof}

\begin{corollary}
Let $(M,H)$, $(\hat{M},\hat{H})$ be as in the setting of Proposition \ref{prop:blow-upT}, and assume there is a generalized complex structure $\mathbb{J}$ on $(M,H)$ $T$-dual to a generalized complex structure $\hat{\mathbb{J}}$ on $(\hat{M},\hat{H})$. Then the $T$-duality from the previous Proposition \ref{prop:blow-upT} sends $\tilde{\mathbb{J}}$, to $\hat{\tilde{\mathbb{J}}}$.
\end{corollary}
\begin{proof}
As $\tilde{\hat{M}}\backslash E \simeq \hat{M}\ \set{\hat{x}}$, we have that the $T$-dual of $\tilde{\mathbb{J}}$ coincides with $\hat{\tilde{\mathbb{J}}}$ outside of $E$. By continuity they must coincide over the entirety of $\tilde{M}_2$.
\end{proof}

\section{Examples}\label{sec:examples}
In this final section we will provide explicit examples of $T$-dual geometric structures on manifolds endowed with standard torus actions. We will define some basic examples, and use these to produce more examples by blowing up.

\subsection{Basic examples}
We start by consider simple examples of generalized complex structures. The analysis heavily depends on the local description in Example \ref{ex:c2}. So recall from there that any of the following two-forms:
\begin{itemize}
\item $F_{1,\pm} = d\theta_1 \wedge d\hat{\theta}_1 \pm d\theta_2 \wedge d\hat{\theta}_2$.
\item $F_{2,\pm} = -d\theta_1 \wedge d\hat{\theta}_2 \pm d\theta_2 \wedge d\hat{\theta}_1$.
\end{itemize}
provides a $T$-duality from $(\cc,0)$ to itself.

\begin{example}[$S^1\times S^3$]
Let $T^2$ act on $S^3 \subset \cc^2$, and extend the action trivially to an action on $S^1\times S^3$. The quotient space $B$ is an annulus. Viewing $S^1\times S^3$ as the Hopf surface $(\cc^2\backslash \set{0})/\mathbb{Z}$, the stable generalized complex structure $\tilde{\rho} = z_1z_2 + dz_1 \wedge dz_2$ on $\cc^2\backslash \set{0}$ descends to a stable generalized complex structure $\rho$ on $S^1\times S^3$. All the two forms $F_{i,\pm}$ are $\mathbb{Z}$-invariant and thus descend to two-forms on $S^1\times S^3$. Therefore, as in Example \ref{ex:c2}, we obtain that $T$-duality induced by the two-form $F_{2,+}$ sends $\rho$ to the standard complex structure on $S^1\times S^3$.

If we instead consider the generalized complex structure induced by $\tilde{\rho}_2 = z_1z_2 + idz_1 \wedge dz_2$, and use the two-form induced by $F_{1,-}$, we find that $\rho_2$ is $T$-dual to itself.
\end{example}

\begin{example}[$\mathbb{C}P^2$]\label{ex:cp2}
We consider $\mathbb{C}P^2$ in the usual fashion as $(\cc^3\backslash \set{0})/\cc$ with torus action $(\lambda_1,\lambda_2)\cdot [z_0:z_1:z_2] = [e^{i\lambda_1}z_0:e^{i\lambda_2}z_1:z_2]$. Consider the complex log forms $\tilde{\zeta_1} = d\log z_0 - d\log z_2$ and $\tilde{\zeta_2} = d\log z_1 - d\log z_2$. As these are basic with respect to the $\cc^*$-action, they descend to complex log forms $\zeta_1,\zeta_2$. We let $\Theta_1 = \Im^* \zeta_1,\Theta_2 = \Im^* \zeta_2$ denote the imaginary parts. Then $(\Theta_1,\Theta_2)$ provides a connection-one form for the $T^2$-action on $\cc P^2$.

Consider another copy of $\mathbb{C}P^2$, and denote the same connection one-form by $(\hat{\Theta}_1,\hat{\Theta}_2)$. Similarly to Example \ref{ex:c2} we now have that any of the following two-forms provides a self $T$-dual for $(\cc P^2,0)$.
\begin{align*}
F_{1,\pm} &= \Theta_1 \wedge \hat{\Theta}_1 \pm \Theta_2 \wedge \hat{\Theta}_2,\\
F_{2,\pm} &= - \Theta_1 \wedge \hat{\Theta}_2 \pm \Theta_2 \wedge \hat{\Theta}_1.
\end{align*}
We will now the action of this $T$-duality to the generalized complex structures on $\cc P^2$.

Consider the holomorphic Poisson structure $\tilde{\pi} = z_0z_1\partial_{z_0}\wedge \partial_{z_1}$. Because this is $\cc^*$-invariant it induces a holomorphic Poisson structure $\pi$ on $\cc P^2$, which is moreover log symplectic. One can show, as in Example \ref{ex:c2}, that $F_{2,+}$ provides a $T$-dual between the corresponding stable generalized complex structure and the standard complex structure on $\cc P^2$. Similarly, consider the holomorphic Poisson structure $\tilde{\pi_2} = iz_0z_1\partial_{z_0}\wedge \partial_{z_1}$, and the induced generalized complex structure $\rho_2$. Then $F_{1,-}$, as in Example \ref{ex:c2}, will provide a $T$-dual between $\rho_2$ and itself.
\end{example}

\begin{example}[$S^2\times S^2$]
Identify $M = S^2\times S^2$ with $\mathbb{C}P^1 \times \mathbb{C}P^1$, and consider the torus action given by $(\lambda_1,\lambda_2)\cdot ([z_0:z_1],[w_0:w_1]) = ([e^{i\lambda_1}z_0:z_1],[e^{i\lambda_2}w_0:w_1])$. Let $\tilde{\zeta}_1 = d\log z_0 - d\log z_1$ and $\tilde{\zeta}_2 = d\log w_0 - d\log w_1$ be complex log-forms on $(\cc^2\backslash \set{0})^2$. As these are basic they induce complex log-forms $\zeta_1,\zeta_2$ on $M$, and we may consider their imaginary parts $\Theta_1 = \Im^*\zeta_1,\Theta_2 = \Im^*\zeta_2$. The form $(\Theta_1,\Theta_2) \in \Omega^2(\elli;\mathfrak{t}^2)$ provides a connection one-form for the torus action on $M$. Let $\hat{M}$ denote another copy of $S^2\times S^2$, and consider the same connection one-form $(\hat{\Theta}_1,\hat{\Theta}_2)$. As in the previous example, the form $F = -\Theta_1 \wedge \hat{\Theta}_2 + \Theta_2 \wedge \hat{\Theta_1}$ provides a self $T$-dual on $(M,0)$. 

In affine coordinates $t = z_0/z_1, s = w_0/w_1$ we consider the holomorphic Poisson structure $\pi = tw\partial_t\partial_w$. In any other set of affine coordinates, this Poisson structure has the same form and hence defines a global holomorphic Poisson structure on $M$, which is moreover log symplectic. As in the previous example, the corresponding stable generalized complex structure is $T$-dual to the standard complex structure on $S^2\times S^2$. And the stable generalized complex structure associated to $\pi = itw\partial_t\partial_w$ is $T$-dual to itself via $F = \Theta_1 \wedge \hat{\Theta}_1 - \Theta_2 \wedge \hat{\Theta}_2$.
\end{example}

We may now blow-up $\mathbb{C}P^2$ to obtain more examples of stable generalized complex structures, which are $T$-dual:
\begin{example}[$n\mathbb{C}P^2 \# m \bar{\mathbb{C}P}^2$]
Let $\tilde{\pi}_1 = z_0z_1\partial_{z_0}\wedge \partial_{z_1}$, and $\tilde{\pi}_2 = iz_0z_1\partial_{z_0}\wedge \partial_{z_1}$ be holomorphic Poisson structures on $\cc^3\backslash \set{0}$, and $\rho_1,\rho_2$ the induced stable generalized complex structures on $\mathbb{C}P^2$, as in Example \ref{ex:cp2}. We may perform consequtive blow-ups of either of these structures, to obtain families of generalized complex structures $\rho_{1,n,m},\rho_{2,n,m}$ on $\elli \rightarrow n \# \cc P^2 m \# \bar{\cc P^2}$. One readily shows that these generalized complex structures on $\elli$, induced generalized complex structures on $n \# \cc P^2 m \# \bar{\cc P^2}$ if and only if $n$ is odd.

We have seen in Example \ref{ex:cp2} that $\rho_1$ can be made $T$-dual to the standard complex structure, and $\rho_2$ dual to itself. By Proposition \ref{prop:blow-upT} it then follows that each $\rho_{2,n,m}$ is $T$-dual to itself, and $\rho_{1,n,m}$ is $T$-dual to the complex structures on $\elli$ described in Example \ref{ex:complexstrucs}.
\end{example}

\begin{remark}
We obtained the families $\rho_{2,n,m}$ also in Theorem 7.5 in \cite{CKW20} via a connected sum procedure. But not the families $\rho_{1,n,m}$. This is of relevance, as we see that their behaviour under $T$-duality is very different.
\end{remark}

\begin{example}[$n\mathbb{C}P^{2k}\# m \bar{\cc P}^{2k}$]
The previous example is in no way restrictive to four dimensions. One may as easily start with the Poisson structure $z_0z_1 \partial_{z_0} \wedge \partial_{z_1} + \cdots + z_{2(k-1)}z_{2k}\partial_{z_{2k-1}}\wedge\partial_{z_{2k}}$ on $\cc^{2k+1}$, and repeat the above to obtain T-dual generalized complex structures on $n\mathbb{C}P^{2k}$. After blowing up at fixed points one again obtains many families of T-dual generalized complex structures on $n \# \cc P^{2k} m \# \bar{\cc P^{2k}}$
\end{example}

\subsection{Non-zero residue elliptic symplectic structures}
Another class of structures which behave well with respect to $T$-duality are elliptic symplectic forms with non-zero elliptic residue. These structures admit a $\cc^*$-invariant normal form:
\begin{lemma}[\cite{W21}]
Let $(M,I_{\abs{D}})$ be a manifold with a smooth elliptic divisor, and let $\omega \in \Omega^2(\elli)$ be an elliptic symplectic form with $\lambda := \Res_q(\omega) \neq 0$. Then there exists a canonical flat connection on $\nu_D$ and a tubular neighbourhood of $D$ such that
\begin{equation}\label{eq:ellilocmod}
\omega = \lambda\rho \wedge \Theta + p^*\omega_D.
\end{equation}
Here $\rho,\Theta$ are the connection-one forms corresponding to the flat connection, as defined in \ref{lem:linebundleconn} and $\omega_D = \iota^*_D\omega \in \Omega^2(D)$. 
\end{lemma}
Conversely, given any complex line bundle $L \rightarrow D$, $\nabla$ a flat connection on $L$, $\lambda \neq 0$ and $\omega_D \in \Omega^2(D)$ symplectic, \eqref{eq:ellilocmod} defines an elliptic symplectic form on $L$.

The form $\omega$ is $S^1$-invariant, and hence we can consider its $T$-dual. It will appear via the following lemma:
\begin{lemma}\label{lem:flatconngcs}
Let $(M,J)$ be a generalized complex manifold, and $L \rightarrow M$ a complex vector bundle, endowed with a flat connection $\nabla$. Then ${\rm tot}(L)$ inherits a natural generalized complex structure.
\end{lemma}
\begin{proof}
Choose a set of trivializations with constant transition functions. On each trivialization use the product of the generalized complex structure on the base and the complex structure of the fiber. Since the transition functions are constant these structures defined on trivializations glue together to give a global structure on the total space of $L$.
\end{proof}

\begin{lemma}
Given $(D,\omega_D)$ a symplectic manifold, $p :L \rightarrow M$ a complex line bundle with flat connection $\nabla$, and $\lambda \in \rr\setminus \set{0}$ then $\omega$ as in \eqref{eq:ellilocmod} is $T$-dual to the type 1 generalized complex structure induced by Lemma \ref{lem:flatconngcs}
\end{lemma}
\begin{proof}
Let $\hat{\Theta}$ denote the same connection one-form, on another copy of $L \rightarrow M$. Then $F = \Theta \wedge \hat{\Theta}$ provides the desired $T$-dual. We have that, using $\tau$ from Proposition \ref{prop:complexiso} that
\begin{equation*}
\tau(e^{i\omega}) = i\lambda(\rho + i\Theta) \wedge e^{i\omega_D}. 
\end{equation*}
Locally, this corresponds to the spinor given by $i\lambda dz \wedge e^{i\omega_D}$, which is precisely the generalized complex structure form Lemma \ref{lem:flatconngcs}.
\end{proof}

\begin{remark}
Given a smooth elliptic divisor $I_{\abs{D}}$, elliptic symplectic forms with zero elliptic residue have a similar $S^1$-invariant normal form as in \eqref{eq:ellilocmod}. Unfortunately, the corresponding $T$-dual generalized complex structure on $\elli$ does not descend to a generalized complex structure on $M$.
\end{remark}

\bibliographystyle{hyperamsplain}
\bibliography{references}

\end{document}